%% file: A-MAIN-higherwente-general-n.tex
\documentclass[11pt,a4paper,leqno]{amsart}   

\usepackage[
  a4paper,
  hmarginratio=1:1, 
  bottom=3.5cm,
]{geometry}

\usepackage[utf8]{inputenc}
\usepackage[T1]{fontenc}
\usepackage{setspace}

\usepackage{amsmath}

\usepackage{empheq}
\usepackage[ english]{babel}
\usepackage[utf8]{inputenc}
\usepackage{lmodern}
\usepackage{csquotes}
\usepackage{tabularx}
\usepackage{mathbbol}
\usepackage{comment}
\usepackage{nccmath}
\usepackage{ragged2e}
\usepackage{nccmath}
\usepackage{tabularx} 

\usepackage{amsfonts}
\usepackage{mdframed}

\usepackage{framed} 
\usepackage{xcolor} 
\colorlet{shadecolor}{pink!20}

\usepackage{float}
\usepackage{latexsym}

\usepackage{amsmath, amssymb,amsthm}
\usepackage{mathtools}  
\usepackage{esint}
\allowdisplaybreaks 

\usepackage{tikz}
\usepackage{comment}
\usepackage{xcolor}
\usepackage{enumitem}
\usepackage{physics}

\usepackage{fancyhdr}
\usepackage{relsize}
\usepackage{lastpage}
\usepackage{hyperref} 

\newcommand{\lorentzl}{L^{(n,q)}} 
\newcommand{\lorentzw}{W^{1,(n,q)}} 
\newcommand{\lorentzone}{L^{(n,1)}} 
\newcommand{\lorentzwone}{W^{1,(n,1)}} 
\newcommand{\lorentzdual}{L^{(\frac{n}{n-1},\infty)}} 
\newcommand{\R}{\mathbb{R}} 
\newcommand{\N}{\mathbb{N}} 

\newcommand{\divv}{\text{div}}

\newcommand{\RNB}[1]{\uppercase\expandafter{\romannumeral#1}}

\usepackage{setspace}
\newtheorem{theorem}{Theorem}[section]

\newtheorem{definition}[theorem]{Definition}

\newtheorem{lemma}[theorem]{Lemma}
\newtheorem{remark}[theorem]{Remark}
\newtheorem{proposition}[theorem]{Proposition}

\newtheorem{corollary}[theorem]{Corollary}
\numberwithin{equation}{section}

\usepackage{microtype} 

\usepackage[
backend=biber,
style=numeric,
sorting=anyt
]{biblatex}
\newcommand{\ignore}[1]{}
\allowdisplaybreaks

\author{Carolin Bayer}
\thanks{Department of Mathematics, ETH Zentrum, CH-8093 Zürich, Switzerland. \newline The author is funded by the Swiss National Fund,
Project SNF $200020\_219429$.}
\address{Department of Mathematics, ETH Zentrum, CH-8093 Zürich, Switzerland}
\email{carolin.bayer@math.ethz.ch }

\begin{document}

\setlength{\jot}{0.5cm}

\title[ Rivière's $GL(m)$-Gauge Construction for Antisymmetric Potentials ]{ Regularity and Rivière's ${GL}(m)$-Gauge Construction for Elliptic Systems with Antisymmetric Potentials in Arbitrary Dimensions.}

\input{A-Project-almost-gauge-critical-case-paper-version}

\printbibliography

\end{document}

%% file: A-Project-almost-gauge-critical-case-paper-version.tex
\maketitle

\begin{abstract}
Let $1 \leq q \le 2$ and denote by $2 \leq q'$ its corresponding conjugate exponent. We prove the continuity of solutions $u \in W^{1,(\frac{n}{n-1},q')}(B^n, \mathbb{R}^m)$ to the critical elliptic system $-\Delta u = \Omega \cdot \nabla u$ in dimension $n \ge 3$, where the potential $\Omega \in L^{(n,q)}(B^n, \mathfrak{so}(m) \otimes \wedge^1)$ is antisymmetric. First, we construct $P \in W^{1,(n,q)}(B^n, \mathrm{SO}(m))$ such that the PDE can be rewritten as $-\operatorname{div}(P^{-1}du) = \ast d\xi \cdot P^{-1}du$, which is nearly a Jacobian structure up to the rotation $P$. Second, we provide a Rivière's $\mathrm{GL}(m)$-Gauge in order to establish a "full" $(A,B)$-conservation law, i.e.  $-\divv(Adu)=d^\ast B \cdot du$. We show that the assumption on $\Omega \in L^{(n,q)} (B^n, \mathfrak{so}(m) \otimes \wedge^1) $ for $q\leq 2$ is optimal.
\end{abstract}

\tableofcontents

\vspace{0.7cm}

\section{Introduction}

\textbf{Motivation of the problem.} The regularity question — whether weak solutions of the harmonic map equation are smooth in the conformally invariant case for $n=2$—  was settled by Hélein \cite{helein}, who used moving frames to handle the nonlinearity of the right-hand side. For maps into symmetric targets, this structure reduces to a Jacobian expression, and continuity follows from the Wente inequality \cite{wente}. More precisely, let $B^2 \subset \mathbb{R}^2$ and let $a, b \in W^{1,2}(B^2)$. Consider the system
\begin{equation}
\begin{cases}
\Delta \varphi = \nabla a \cdot \nabla^\perp b & \text{on } B^2, \\
\varphi = 0 & \text{on } \partial B^2. \label{eq:wente}
\end{cases}
\end{equation}
Wente discovered improved regularity properties for the solution $\varphi$ of \eqref{eq:wente} given by
\begin{equation}
\|\varphi\|_{L^\infty(B^2)} + \|\nabla \varphi\|_{L^2(B^2)} \le C \|\nabla a\|_{L^2(B^2)} \|\nabla b\|_{L^2(B^2)}.
\end{equation}

A generalization of these better integrability estimates goes back to the compensated compactness theory of Coifman, Lions, Meyer, and Semmes \cite{clms} for functions that lie in Hardy space $\mathcal{H}^1$ rather than merely $L^1$. They further showed that solutions of \eqref{eq:wente} satisfy
\begin{equation} \label{eq:clms-equation}
\|\nabla \varphi\|_{L^{2,1}(B^2)} + \|\nabla^2 \varphi\|_{L^1(B^2)} \le C \|\nabla a\|_{L^2(B^2)} \|\nabla b\|_{L^2(B^2)}.
\end{equation}
Bethuel showed a similar version \cite{bethuel}, namely
\begin{equation}
\|\nabla \varphi\|_{L^2(B^2)} \le C \|\nabla a\|_{L^{2,\infty}(B^2)} \|\nabla b\|_{L^2(B^2)}. \label{eq:bethuel}
\end{equation}

 The full geometric reach of this cancellation structure was discovered by Rivière \cite{conservation-law-conf-invariant} for $n=2$, who proved that any conformally invariant elliptic system of the form  $-\Delta u = \Omega \cdot \nabla u$ with antisymmetric potential $\Omega \in L^2$
 can be rewritten as a so called "$(A,B)$-conservation law" / "Rivière's $\mathrm{GL}(m)$-Gauge", yielding continuity without any assumption on the target.

\vspace{0.5cm}

The question naturally arises as to whether this structure provides improved regularity estimates, even in higher dimensions. Let $n\geq 3$. Rivière raised the question whether it remains true or not, that any solution \(u \in W^{1,n}(B^n, \mathbb{R}^m)\) of
\begin{equation} 
    -\operatorname{div}\!\left(|\nabla u|^{n-2}\nabla u\right) = |\nabla u|^{n-2}\,\Omega \cdot \nabla u, \label{eq:conjecture-tristan}
\end{equation}
where $\Omega\in L^n(B^n, \mathfrak{so}(m)\otimes \wedge^1 )$ is also continuous. Recently, Martino and Schikorra and independently Schlagenhauf provided a negative answer to this question by giving a counterexample \cite{answer-martino-schikorra}, \cite{answer-schlagenhauf}. It turns out that the natural extension of these results leads to the so-called Lorentz spaces, which were introduced in Section \ref{sec:preliminary} and play the same crucial role in our result as they do in, for example,  \cite{dorian,tian}.

\vspace{0.5cm}

\textbf{Formulation of the problem.}
Let now $n\geq 3$ and $m\in \N$. Fix $1 \leq q \leq 2$ and $q'= \frac{q}{q-1} \geq 2$ its conjugate exponent. Let $u \in W^{1, (\frac{n}{n-1},q')}(B^n, \R^m)$ be a solution of the system
\begin{equation} \begin{aligned}
   - \Delta u= \Omega \cdot \nabla u \quad \text{ in } \mathcal{D}'(B^n), \label{eq: model-pde}
\end{aligned} \end{equation}
where $\Omega \in L^{(n,q)}(B^n, \mathfrak{so}(m)\otimes \wedge^1)$ is anti-symmetric.
The right hand side of \eqref{eq: model-pde} becomes critical due to $L^{(\frac{n}{n-1},q')} \cdot L^{(n,q)} \hookrightarrow L^1$ and it corresponds to the endpoint case of the classical Calderón-Zygmund theory. In this setting, we establish better regularity, that is, the continuity of solutions of \eqref{eq: model-pde}.

\vspace{0.3cm}

The assumption $\Omega \in L^{(n,q)}$ arises naturally in the proof (see \eqref{eq:motivation-q-q-prime}), as it is required to simultaneously obtain a bound on both – the potential $d\xi$ and the gauge $dP$.
It is noteworthy that in the endpoint case $q=2$, the assumption is the same as the one of  Martino and Schikorra in their construction of a gauge for antisymmetric potentials for rotated $n$-harmonic systems \cite{dorian}. At time of creation, the author was not aware of the publication of Guo, Xiang \cite{guo}, which shows that the technique of the Rivière–Gauge construction also works in the case of $\Omega \in L^{(n,2)}$  and coincides with our endpoint case.

\vspace{0.5cm}

We present two approaches to establishing the continuity of $u \in W^{1,(\frac{n}{n-1},q')}$.
In a first version given in Section \ref{sec:version1}, we employ duality estimates below the Lorentz exponent, in particular  $L^{(\frac{n}{n-1},\infty)}$, which is only slightly larger than $L^{(\frac{n}{n-1},q')}$ and has already played a crucial role in the work of Bethuel \cite{bethuel} (see \eqref{eq:bethuel}) and \cite{armin-gauge}. This approach also illustrates that one can circumvent the need to establish the more challenging full $(A,B)$-conservation laws. 

This first approach  is inspired by Rivière and Struwe \cite{struwe-riviere} for systems of the type \eqref{eq: model-pde}, but under their assumption that $u\in W^{1,2}(B^n,\R^m)$ and  $\Omega \in L^2(B^n, \mathfrak{so}(m) \otimes \wedge^1)$. In their setting, the system \eqref{eq: model-pde} becomes supercritical in dimensions $n\geq 3$ and they have to assume additional Morrey assumptions in order to establish an "almost Jacobian" structure. 
\vspace{0.5cm}
We now give the  main theorem of the first approach.
\begin{theorem}\label{thm:main-thm}
Let $m \in \mathbb{N}$. Let $q \le 2$ and denote by $q' = \frac{q}{q - 1}\geq 2$ its conjugate exponent. For every $n\geq 3$, there exists $\varepsilon_0(n)>0$ such that for every weak solution $u \in W^{1, (\frac{n}{n-1},q')}(B^n, \R^m)$ of the equation \eqref{eq: model-pde}
  and for every  $\Omega \in L^{(n,q)}(B^n, \mathfrak{so}(m) \otimes \wedge^1)$  satisfying the smallness assumption 
 \begin{align}
        \Vert \Omega \Vert_{L^{(n,q)}(B^n)} < \varepsilon_0(n), \tag{A}  
    \end{align}
then there exist $\alpha \in (0,1)$ such that  $u\in C^{0,\alpha}_\mathrm{loc}(B^n,\R^m)$. 
\end{theorem}

\vspace{0.5cm}
In a second version given in Section \ref{sec:version2}, we show that it is indeed  possible to construct a "full" $(A,B)$-conservation law/ Rivière's $\mathrm{GL}(m)$-Gauge to obtain the same result. 
More precisely, instead of transforming the system by \(P^{-1}\), as was done in the first approach, the PDE is now transformed by
\(A \in L^\infty \cap W^{1,(n,q)}(B^n,\mathrm{GL}(m))\) and we obtain
\begin{equation}
    \begin{aligned}
        -\operatorname{div}(A\, du)
        &= -dA \cdot du + A \Omega \cdot du
        = \underbrace{(-dA + A\Omega)}_{=:(i)}  \cdot du.
    \end{aligned}
\end{equation}
Further, we show the existence of $B \in W^{1,(n,q)}(B^n, \R^{m \times m} \otimes \wedge ^2)$  in order to write \((i)\) as a purely co-exact expression, so that the right-hand side becomes a "full" Jacobian:
\begin{align}
    -\operatorname{div}(Adu)= \underbrace{d^\ast B}_{\in L^{(n,q)}}  \cdot \underbrace{du}_{\in L^{(\frac{n}{n-1},q')}}.
\end{align}

The following theorem is the main result of this second approach and leads to optimal regularity.

\begin{theorem}\label{thm:main-theorem-2}
Let $n\geq 3$ and $m \in \mathbb{N}$. Let $q \le 2$ and denote by $q' = \frac{q}{q - 1}\geq 2$ its conjugate exponent. There exists $\varepsilon_0(n)>0$ such that for every weak solution $u \in W^{1, (\frac{n}{n-1},q')}(B^n, \R^m)$ of the equation \eqref{eq: model-pde}
  and for every  $\Omega \in L^{(n,q)}(B^n, \mathfrak{so}(m) \otimes \wedge^1)$  satisfying the smallness assumption 
 \begin{align}
        \Vert \Omega \Vert_{L^{(n,q)}(B^n)} < \varepsilon_0(n), \tag{A}  
    \end{align}

  we then have $u\in W^{2,p}_\mathrm{loc}(B^n,\R^m)$ for any $p<n$ and hence $C^{0,\alpha}_\mathrm{loc}(B^n,\R^m)$ for any $\alpha \in (0,1)$. \\ 
 Moreover, there exists a $(A,B)$–Gauge with 
 $A \in L^\infty\cap W^{1,(n,q)}(B^n, \R^{m\times m})$ and $B\in W^{1,(n,q)}(B^n, \R^{m\times m} \otimes \wedge^2)$ satisfying
    \begin{equation}
        -d^\ast B \coloneqq  dA - A \Omega. \label{eq:def-b-star}
    \end{equation}
    such that the weak solution $u\in W^{1, (\frac{n}{n-1},q')}(B^n,\R^m)$ of \eqref{eq: model-pde} satisfies the following conservation law
\begin{equation}
    \begin{aligned}
        d(\ast A du + (-1)^{n-1} (\ast B)\wedge du)=0. \label{eq:conservation-law}
    \end{aligned}
\end{equation}
\end{theorem}
\begin{remark}
    We point out that Theorem \ref{thm:main-theorem-2} immediately implies Theorem \ref{thm:main-thm}. 
\end{remark}

\textbf{Novelty of this project.}
 The novelty of this project lies in developing a more refined understanding of critical systems with antisymmetric potentials in higher dimensions. 
The paper relies heavily on the following type of estimate:
\begin{equation}
    \begin{aligned}
        \Vert \varphi \Vert_{L^\infty(B^n)}+ \Vert d\varphi \Vert_{L^{(n,1)}(B^n)}  + \Vert \nabla^2 \varphi \Vert_{L^{(\frac{n}{2},1)}(B^n)}   \lesssim \Vert da \Vert_{L^{(n,q)}(B^n)} \Vert db \Vert_{L^{(n,q')}(B^n)},
    \end{aligned} \label{eq:high-wente}
\end{equation}
where $\varphi \in W^{1,(n,1)}(B^n,\R^m)$ is a solution of 
\begin{equation}
\begin{cases}
\Delta \varphi = \ast (da \wedge  db) & \text{on } B^n, \\
\varphi = 0 & \text{on } \partial B^n, \label{eq:wente2}
\end{cases}
\end{equation}
and $a\in W^{1,(n,q)}(B^n,\R^m)$ and $b\in W^{1,(n,q')}(B^n, \wedge^{n-2})$.
Estimate \eqref{eq:high-wente} can be interpreted as a "high-dimensional Wente" and  should be compared with \eqref{eq:clms-equation}. 
\vspace{0.5cm}

In Section \ref{sec:counterexample} we show that the assumption on the integrability of $\Omega\in L^{(n,q)}(B^n, \mathfrak{so}(m) \otimes \wedge^1)$ for $q\leq 2$  is indeed optimal. More precisely, one cannot assume a milder assumption on $\Omega \in L^{n,q}(B^n, \mathfrak{so}(m) \otimes \wedge^1)$ for some $q>2$ in order to conclude continuity on $u$.

\vspace{0.5cm}

\section{Preliminary}\label{sec:preliminary}

We provide a brief introduction to Lorentz spaces, which arise naturally as refinements of standard Lebesgue spaces. For further reading, we refer to \cite[p.48]{grafakos}. 

\vspace{0.5cm}

Let $\Omega \subset \mathbb{R}^n$ be an open set. For any measurable function $f: \Omega \to \mathbb{R}$, we define its decreasing rearrangement $f^\ast: [0, \infty) \to [0, \infty)$ as
\begin{equation}
    f^\ast(t) := \inf \{\lambda \ge 0 : |\{x \in \Omega : |f(x)| > \lambda\}| \le t\} \quad \text{for } t > 0,
\end{equation}
and its average as 
\begin{equation}
f^{\ast\ast}(s) := \frac{1}{s} \int_0^s f^\ast(t) dt, \quad \text{ for } s>0.
\end{equation}

Given exponents $p \in [1, \infty)$ and $q \in [1, \infty]$, the Lorentz quasi-norm is defined as
\begin{equation}
    [f]_{L^{(p,q)}(\Omega)} := 
    \begin{cases} 
     \Bigg( \displaystyle  \int_0^\infty t^{\frac{q}{p}}   f^\ast(t)^q \frac{dt}{t} \Bigg)^{\frac{1}{q}} & \text{if } q < \infty, \\
       \displaystyle  \sup_{t > 0} t^{\frac{1}{p}} f^\ast(t) & \text{if } q = \infty.
    \end{cases} \label{eq:quasinorm}
\end{equation}
For  $p \in (1, \infty)$ and $ q \in [1,\infty]$, replacing $f^\ast$ with $f^{\ast\ast}$ in the  definition \eqref{eq:quasinorm} yields an equivalent norm to \eqref{eq:quasinorm}, denoted by $\|f\|_{L^{(p,q)}(\Omega)}$ (see \cite[Thm. 4.6, Chap. 4]{bennett}).
\begin{definition}
    Let $k \in \mathbb{N}$ and $\Omega \subset \mathbb{R}^n$ be an open set. For $1 < p < \infty$, $1 \le q \le \infty$, we define the Soboloev-Lorentz space as
\begin{equation}
W^{k,(p,q)}(\Omega) := \left\{f \in L^{(p,q)}(\Omega) : \partial^\alpha f \in L^{p,q}(\Omega)  \text{ for each } 0 \le |\alpha| \le k\right\}.
\end{equation}
\end{definition}
It holds $L^p(\Omega) = L^{(p,p)}(\Omega)$.  
If $|\Omega| < \infty$ and $1 \le q < p < r \le \infty$, we have  \cite[Prop.4.2, Chap. 4]{bennett}
\begin{equation}
   \underbrace{L^{(r,r)}}_{ =L^r(\Omega)}(\Omega) 
   \overset{r>p}{\subset} L^{(p,r)}(\Omega)   \overset{r<p}{ \subset }  \underbrace{L^{(p,p)}(\Omega)}_{=L^p(\Omega)} 
  \overset{p<q}{ \subset }  L^{(p,q)}(\Omega) 
  \overset{ p>q }{\subset} \underbrace{  L^{(q,q)}(\Omega)}_{=L^q(\Omega)}.
\end{equation}

Several classical functional inequalities can be generalised to Lorentz spaces.

\begin{theorem}[Hölder's Inequality for Lorentz Spaces.] \label{thm:holder-lorentz} 
Assume $f \in L^{(p_1,q_1)}(\Omega)$ and $g \in L^{(p_2,q_2)}(\Omega)$ for   $1 < p_1, p_2, p < \infty$ and let $1 \le q_1, q_2, q \le \infty$, that satisfy
\begin{equation}
    \begin{aligned}
        \frac{1}{p_1} + \frac{1}{p_2} = \frac{1}{p}, \qquad \text{and } \qquad 
        \frac{1}{q_1} + \frac{1}{q_2} \ge \frac{1}{q}.
    \end{aligned}
\end{equation}
Then $f g \in L^{(p,q)}(\Omega)$ and satisfies
\begin{equation}
    \|fg\|_{L^{(p,q)}} \le C(p_1,p_2,q_1,q_2) \|f\|_{L^{(p_1,q_1)}} \|g\|_{L^{(p_2,q_2)}}.
\end{equation} This result  extends to the case where $(p_1, q_1) = (\infty, \infty)$, reducing to the standard $L^\infty$ bound.
\end{theorem}
\begin{proof}
    See \cite[Thm. 3.4]{neil}
\end{proof}

\begin{theorem}[Young's Inequality for Convolutions.] \label{thm:young-ineq}
 Assume $f \in L^{(p_1,q_1)}(\R^n)$ and $g \in L^{(p_2,q_2)}(\R^n)$ for   $1 < p_1, p_2, p < \infty$ and let $1 \le q_1, q_2, q \le \infty$, that satisfy
\begin{equation}
    \begin{aligned}
        \frac{1}{p_1} + \frac{1}{p_2} = 1+\frac{1}{p}, \qquad \text{and } \qquad 
        \frac{1}{q_1} + \frac{1}{q_2} \ge \frac{1}{q}.
    \end{aligned}
\end{equation} 
Then $f \ast g \in L^{(p,q)}(\mathbb{R}^n)$ and satisfies 
\begin{equation}
    \|f \ast g\|_{L^{(p,q)}} \le C \|f\|_{L^{(p_1,q_1)}} (p_1,p_2,q_1,q_2)\|g\|_{L^{(p_2,q_2)}}.
\end{equation}

\end{theorem}
\begin{proof}
    See \cite[Thm. 2.6]{neil}.
\end{proof}

\begin{corollary}[Improved Sobolev Embeddings]\label{cor:improved-sobolev}
    Let $n \ge 2$, $1 \le p < n$, and $1 \leq q \leq \infty$. Define the Sobolev conjugate exponent $p^*= \frac{np}{n-p}$.
    Then the Sobolev-Lorentz space $W^{1,(p,q)}(\mathbb{R}^n)$ embeds continuously into the Lorentz space $L^{(p^*, q)}(\mathbb{R}^n)$:
    \begin{equation}
        W^{1,(p,q)}(\mathbb{R}^n) \hookrightarrow L^{p^*, q}(\mathbb{R}^n).
    \end{equation}  
    In particular, for standard Sobolev spaces where $q=p$, we have $W^{1,p}(\mathbb{R}^n) \hookrightarrow L^{p^*, p}(\mathbb{R}^n)$, and for the endpoint $p = 1$,
    \begin{equation}\label{eq:poornima}
        W^{1,1}(\mathbb{R}^n) \hookrightarrow L^{\frac{n}{n-1}, 1}(\mathbb{R}^n).
    \end{equation}
\end{corollary}
\begin{proof}
   See \cite[p. 294, Thm. 4.1]{peetre}, \cite[Thm. 2.6]{neil}. The case $p=1$ can be find in \cite{poornima}
\end{proof}

\begin{proposition}
    Let $\Omega \subset \mathbb{R}^n$  be a  bounded domain with Lipschitz boundary. Let $1 \le p < n$ and $1 \le q \le \infty$, and  $p^* = \frac{np}{n-p}$ its  conjugate Sobolev exponent. For every $u \in W^{1,(p,q)}(\Omega)$, we have
\begin{equation}\label{eq:poincare_lorentz_critical}
    \|u - \bar{u}_\Omega\|_{L^{p^*,q}(\Omega)} \le C \|\nabla u\|_{L^{p,q}(\Omega)},
\end{equation}
where $\bar{u}_\Omega = \frac{1}{|\Omega|} \int_\Omega u(x) \, dx$ and $C = C(n, p, q, \Omega) > 0$ is a constant independent of $u$.
\end{proposition}

\begin{proposition}\label{prop:continuity}
   Let  $f\in L_{\mathrm{loc}}^1(\R^n)$ with $\nabla f \in L^{(n,1)}(\R^n)$. Then $f$ is equivalent to a continuous function and satisfies 
   \begin{equation}
       \Vert f-c\Vert_{L^\infty} \leq C \Vert \nabla f \Vert_{L^{(n,1)}},
   \end{equation}
   where $c=c(f)$ a suitable constant. 
\end{proposition} 

The corresponding proofs can be found in \cite{bennett, neil, stein}.

\vspace{0.7cm}
\section{Construction of a Gauge $P \in W^{1,(n,q)}(B^n,\mathrm{SO}(m))$.}
The following Lemma provides a suitable gauge that, due to the better integrability assumption of
$\Omega \in L^{(n,q)}$ compared to the gauge in \cite{conservation-law-conf-invariant}, also inherits better integrability. The proof strategy remains the same as in \cite{conservation-law-conf-invariant} and \cite{struwe-riviere}. 

\begin{lemma}\label{lm:constructing-gauge-p-original-space}
 Let $1 \leq q \leq 2$. There exists $\varepsilon_0>0$ and $C>0$ such that for any $\Omega \in  L^{(n,q)}(B^n, \mathfrak{so}(m)\otimes \wedge^1)$  satisfying 
  \begin{align}
      \Vert \Omega \Vert_{\lorentzl(B^n)} \leq C \varepsilon_0, \label{eq:smallness-for-gauge}
  \end{align}
    then there exists $P \in W^{1,(n,q)}(B^n,\mathrm{SO}(m))$ and $\xi \in W^{1,(n,q)}(B^n, \mathfrak{so}(m) \otimes \wedge^{n-2})$ such that 
    \begin{equation} \begin{aligned}
        \begin{cases}
            \ast d\xi = P^{-1} dP + P^{-1} \Omega P & \text{ in } B^n, \\
            d\ast\xi = 0 & \text{ in } B^n, \\
            \xi = 0 & \text{ on } \partial B^n.
        \end{cases} \label{eq:gauge-system}
    \end{aligned} \end{equation} 
    Moreover, the following a--priori bound holds: 
    \begin{equation} \begin{aligned}
       \Vert dP \Vert_{L^{(n,q)}(B^n)} + \Vert d \xi \Vert_{L^{(n,q)}(B^n)} \leq C \Vert \Omega \Vert_{L^{(n,q)}(B^n)} \leq C \varepsilon_0. \label{eq:inequality-gauge-original}
    \end{aligned} \end{equation} 
\end{lemma}

We follow the approach of Uhlenbeck in \cite[Lemma 2.7, Lemma 2.8]{uhlenbeck}. As a first step, we prove the Lemma for a slightly more regular space, see Lemma  \ref{lm:better-regularity-space}. This makes the map $\lambda \mapsto P^{-1} \lambda P$ smooth. Using an approximation argument, we derive Lemma \ref{lm:constructing-gauge-p-original-space} as a consequence of Lemma \ref{lm:better-regularity-space}.

\begin{proof}[Proof of Lemma \ref{lm:constructing-gauge-p-original-space}.] 
   We will show that Lemma \ref{lm:constructing-gauge-p-original-space} follows directly from Lemma \ref{lm:better-regularity-space}. 

Let $\Omega \in L^{(n,q)}(B^n,\mathfrak{so}(m)\otimes \wedge^1)$. We assume that $\Vert \Omega \Vert_{L^{(n,q)}(B^n)} < \varepsilon_0$ for some sufficiently small $\varepsilon_0 > 0$ and fix $\varepsilon_0$ later.

Let $\frac{n}{2}<p<n$ as in Lemma \ref{lm:better-regularity-space} and recall that then $W^{1,(p,1)}\hookrightarrow L^{(n,q)}$. 
Due to the density of $C_c^\infty(B^n)$ in $L^{(n,q)}$, there exists a sequence $(\Omega_k)_{k\in \mathbb{N}} $ in $ W^{1,(p,1)}(B^n)$ such that 
\[
\Omega_k \to \Omega \quad \text{strongly in } L^{(n,q)}. 
\]
We can construct this sequence using standard mollifiers, setting $\Omega_k \coloneqq \eta_k \ast \Omega$. By Young's convolution inequality, we have:
\[
\Vert \Omega_k \Vert_{L^{(n,q)}} \leq \Vert \eta_k \Vert_{L^1} \Vert \Omega \Vert_{L^{(n,q)}} = \Vert \Omega \Vert_{L^{(n,q)}} < C\varepsilon_0.
\]
By choosing $C\varepsilon_0 \leq \varepsilon(n,m)$, we ensure that $\Vert \Omega_k \Vert_{L^{(n,q)}} < \varepsilon(n,m)$ for all $k$, satisfying the smallness assumption required for Lemma \ref{lm:better-regularity-space}.

Applying Lemma \ref{lm:better-regularity-space} to each $\Omega_k$, there exist sequences $P_k \in W^{2,(p,1)}(B^n,\mathrm{SO}(m))$ and $\xi_k \in W^{2,(p,1)}(B^n, \mathfrak{so}(m) \otimes \wedge^{n-2})$ satisfying the gauge system \eqref{eq:gauge-system-better}, as well as the uniform bounds \eqref{eq:apriori-bound-better-lemma1} and \eqref{eq:apriori-bound-better-lemma2}. 

Because the bounds in \eqref{eq:apriori-bound-better-lemma2} are uniform with respect to $k$, we can extract weakly converging subsequences (the index is again denoted by $k$) such that $P_k \rightharpoonup P$ and $\xi_k \rightharpoonup \xi$ weakly in $W^{2,(p,1)}$. Passing to the limit as $k \to \infty$, we conclude the proof of Lemma \ref{lm:constructing-gauge-p-original-space}.
\end{proof}

\vspace{0.7cm}
\section{Proof of Theorem \ref{thm:main-thm} –  an "almost" Jacobian structure.} \label{sec:version1}

\begin{proof}[Proof of Theorem \ref{thm:main-thm}] \label{proof:proof-of-main-theorem}

Since continuity is a local property, it suffices to prove continuity on the ball $B_{1-\tau_0}(0)$ for an arbitrary $\tau_0 \in (0,1)$. Since the $L^{(n,q)}$-norm is absolute continuous, we can find a uniform radius $r_0 > 0$ such that for any $r < r_0$ and any $x \in B_{1-\tau_0}(0)$, we have $B_r(x) \subset B^n$ and
\begin{align}
    \Vert \Omega \Vert_{\lorentzl (B_r(x))}\leq C\varepsilon_0. \label{eq:assumption-ball}
\end{align}
During the proof, it is possible to adopt $\varepsilon_0$ as small as is necessary.
By Lemma \ref{lm:constructing-gauge-p-original-space}, this yields a local gauge $P \in \lorentzw(B_r(x), \mathrm{SO}(m))$ and $\xi \in \lorentzw(B_r(x), \mathfrak{so}(m)\otimes \wedge^{n-2})$ for which \eqref{eq:gauge-system} holds on $B_r(x)$.

Applying this gauge transformation to $du$, we deduce:
\begin{equation} \label{eq:almost-div-curl}
\begin{aligned}
 - d^\ast\!(P^{-1} d u) &= - d P^{-1} \cdot du + P^{-1} \Omega du \\
 &= (P^{-1} d P + P^{-1} \Omega P) P^{-1} du \\
 &= \ast d\xi \cdot P^{-1} d u,
\end{aligned}
\end{equation} 
where we used the identity $d(P^{-1}P) = 0$, which yield $-d P^{-1} P = P^{-1} d P$. In analogy with \cite{struwe-riviere}, equation \eqref{eq:almost-div-curl} defines an ``almost'' Jacobian structure.

\vspace{0.5cm}

In this approach, we apply a duality argument inspired by \cite[Thm. 3.5]{lecture-notes-tristan} and \cite[p. 459]{struwe-riviere}. Let $\delta \in (0,1)$ and consider an arbitrary 1-form $\phi \in \lorentzone (B_{\delta r}(x), \wedge^1)$ with $\Vert \phi \Vert_{L^{(n,1)}(B_{\delta r}(x))}\leq 1$. Further, define the extension by zero $\widetilde{\phi} \coloneqq \phi \chi_{B_{\delta r}(x)}$, so that $\Vert \widetilde{\phi} \Vert_{L^{(n,1)}(B_r(x))} \leq 1$. 
\vspace{0.2cm}

Consider the Hodge decomposition (see \cite[Cor. 10.5.1]{iwaniec-hodge-book} for further details)
\begin{equation}
   \widetilde{\phi} = d\varphi + d^\ast \psi + \mathfrak{h} \qquad \text{in } B_r(x),
\end{equation}
where $d\varphi \in dW_T^{1,(n,1)}(B_r(x))$, $d^\ast \psi \in d^\ast W_N^{1,(n,1)}(B_r(x), \wedge^{2})$,\footnote{This notation was introduced in \cite[p.~38]{iwaniec-hodge}. \\
$dW_T^{1,(r,s)}(B_r(x)) \coloneqq \{ d\varphi \mid \varphi \in W^{1,(r,s)}(B_r(x)), \; \varphi|_{\partial B_r(x)} = 0 \}$, \\
$d^\ast W_N^{1,(r,s)}(B_r(x),\wedge^2) \coloneqq \{ d^\ast\psi \mid \psi \in W^{1,(r,s)}(B_r(x), \wedge^2), \; \ast\psi|_{\partial B_r(x)} = 0 \}$.} and $\mathfrak{h} \in C^\infty(B_r(x), \wedge^1)$ is a harmonic 1-form satisfying $d\mathfrak{h}=0$, $d^\ast \mathfrak{h}=0$. Moreover, it holds:
\begin{align}
    \Vert d\varphi \Vert_{L^{(n,1)}(B_r(x))} + \Vert d^\ast \psi \Vert_{L^{(n,1)}(B_r(x))} + \Vert \mathfrak{h} \Vert_{L^{(n,1)}(B_r(x))} \leq C \Vert \widetilde{\phi} \Vert_{L^{(n,1)}(B_r(x))} \leq C. \label{eq:estimate-by-one}
\end{align}
We note briefly that this extension is required to ensure interior estimates for the harmonic part of $\widetilde{\phi}$ at a later stage. In addition, it enables us to combine all individual components in the final step.

\vspace*{0.5cm}

\begin{center}
    \textsc{Morrey decay estimate for $\Vert \nabla u \Vert_{\lorentzdual(B_{\delta r}(x))}$} 
\end{center}

Since $P \in \mathrm{SO}(m)$, the norm is preserved, i.e., $|P^{-1} \nabla u| = |\nabla u|$. Using the duality between $\lorentzone$ and $\lorentzdual$, we can write:
\begin{equation} \label{eq:containing-step-by-step}
\begin{aligned}
\Vert \nabla u \Vert_{\lorentzdual(B_{\delta r}(x))} 
&= \Vert P^{-1} \nabla u \Vert_{\lorentzdual(B_{\delta r}(x))} \\
&= \sup_{ \substack{ \phi \in \lorentzone(B_{\delta r}(x), \wedge^1) \\ \Vert \phi \Vert_{\lorentzone(B_{\delta r}(x))} \leq 1}} \int_{B_{\delta r}(x)} P^{-1} du \cdot \phi \; dvol_n.
\end{aligned} 
\end{equation} 

Fix an arbitrary test form $\phi \in \lorentzone(B_{\delta r}(x),\wedge^1)$ with $\Vert \phi \Vert_{\lorentzone(B_{\delta r}(x))} \leq 1$ with the extension $\widetilde{\phi}$ on $B_r(x)$ as explained above. Substituting the Hodge decomposition $\widetilde{\phi} = d\varphi + d^\ast \psi + \mathfrak{h}$, we obtain:
\begin{equation}
\begin{aligned}
\int_{B_{\delta r}(x)} P^{-1} du \cdot \phi \; dvol_n 
&= \int_{B_r(x)} P^{-1} du \cdot \widetilde{\phi} \; dvol_n \\
&= \int_{B_r(x)} P^{-1} du \cdot (d\varphi + d^\ast \psi + \mathfrak{h}) \; dvol_n. 
\end{aligned}
\end{equation}

We will estimate each of these three terms separately. In the last step, we take the supremum over all such $\phi \in \lorentzone(B_{\delta r}(x),\wedge^1)$ and we conclude by combining them.

\vspace{0.5cm}

\textbf{Estimate containing $d\varphi$.}

Define $\displaystyle \bar{u}_{B_r(x)} \coloneqq \fint_{B_r(x)} u$. 
Using Stokes' Theorem and the fact that $\varphi|_{\partial B_r(x)}=0$, the boundary term vanishes, yields:

\begin{equation} \label{eq:containing-dvarphi}
\begin{aligned}
&\int_{B_r(x)} d\varphi \wedge \ast P^{-1} du 
= \int_{B_r(x)} d(\varphi  \ast P^{-1}du) - \int_{B_r(x)} \varphi  d(\ast P^{-1} du) \\
&= \int_{B_r(x)} d \xi \wedge (\varphi P^{-1}) du \\
&= (-1)^{n-1}\int_{B_r(x)} d \big( d\xi  \varphi P^{-1} (u- \bar{u}_{B_r(x)}) \big) -  \int_{B_r(x)} d\xi \wedge d(\varphi P^{-1}) (u-\bar{u}_{B_r(x)}) \\
&= - \int_{B_r(x)} d\xi \wedge d(\varphi P^{-1}) (u-\bar{u}_{B_r(x)}) 
\end{aligned}
\end{equation}
This implies in particular  
\begin{equation}
\begin{aligned}
\Bigg\vert \int_{B_r(x)} d\varphi \wedge \ast P^{-1} du  \Bigg\vert &\lesssim \Vert d\xi \Vert_{L^{(n,q)}(B_r(x))} \Big( \Vert d\varphi \Vert_{L^{(n,q')}}  + \Vert \varphi \Vert_{L^\infty} \Vert dP\Vert_{L^{(n,q')}}\Big) \\ 
& \hspace{4cm} \cdot \Vert u- \bar{u}_{B_r(x)} \Vert_{L^{(\frac{n}{n-2},\infty)}} \\
&\leq C'\varepsilon_0 (1+ C\varepsilon_0) \Vert \nabla u\Vert_{\lorentzdual(B_r(x))}, 
\end{aligned}
\end{equation}
where we used $d(d\xi)=0$,  $ \lorentzwone \hookrightarrow  W^{1,(n,q)}\hookrightarrow W^{1,(n,q')}$ due to 
\begin{equation}
    \begin{aligned}
        q \leq q' \quad \Leftrightarrow \quad  \frac{1}{q} \geq \frac{1}{q'}=1-\frac{1}{q} \quad \Leftrightarrow \quad 1 \leq q \leq 2, \label{eq:motivation-q-q-prime}
    \end{aligned}
\end{equation}
the embedding $\lorentzwone \hookrightarrow L^\infty$ (see Prop. \ref{prop:continuity}), estimate \eqref{eq:inequality-gauge-original} together with \eqref{eq:assumption-ball}, and the Lorentz-Hölder exponent relations (see Cor. \ref{cor:improved-sobolev}):
\begin{equation}
\frac{1}{n} + \frac{1}{n} + \frac{n-2}{n} = 1, \qquad \frac{1}{q} + \frac{1}{q'} + \frac{1}{\infty} = 1.
\end{equation}

\vspace*{0.5cm}

\textbf{Estimate containing $d^\ast \psi$.}

Analogously to the above, using Stokes' Theorem together with the boundary condition $\ast \psi |_{\partial B_r(x) }=0$, we obtain:
\begin{equation} \label{eq:containing-dpsi}
\begin{aligned}
 \int_{B_r(x)} P^{-1} &du \wedge d\ast \psi    
 = \int_{B_r(x)} \Big(d(P^{-1}(u-\bar{u}_{B_r(x)})) - dP^{-1}(u-\bar{u}_{B_r(x)}) \Big) \wedge d\ast \psi \\
 &= -\int_{B_r(x)} d \Big( d\big( P^{-1} (u-\bar{u}_{B_r(x)}) \big) \wedge \ast \psi \Big) - \int_{B_r(x)} dP^{-1}(u-\bar{u}_{B_r(x)}) \wedge d\ast \psi \\
 &= -\int_{B_r(x)} dP^{-1}(u-\bar{u}_{B_r(x)}) \wedge d\ast \psi.
\end{aligned}
\end{equation}
 This implies in particular 
 \begin{equation}
     \begin{aligned}
     \Bigg\vert  \int_{B_r(x)} P^{-1} du \wedge d\ast \psi    \Bigg\vert 
 &\lesssim \Vert dP \Vert_{L^{(n,q)}(B_r(x))}\Vert u-\bar{u}_{B_r(x)} \Vert_{L^{(\frac{n}{n-2},\infty)}(B_r(x))} \Vert d^\ast \psi \Vert_{L^{(n,q')}(B_r(x))} \\
 &\overset{\eqref{eq:poincare_lorentz_critical}}{\leq} C''\varepsilon_0 \Vert \nabla u\Vert_{\lorentzdual(B_r(x))},
\end{aligned}
\end{equation}
where we used the embedding (see Theorem  \ref{thm:holder-lorentz})
\begin{equation}
    \begin{aligned}
        L^{(n,q)} \cdot L^{(n,q')} \hookrightarrow L^{(n/2,1)}.
    \end{aligned}
\end{equation}

\begin{remark}
We briefly explain why the assumption $\Omega \in L^{(n,q)}$ is indeed necessary and optimal. In \eqref{eq:containing-dpsi}, it would actually be enough to assume only $dP \in L^n$, because $d^\ast \psi \in L^{(n,\frac{n}{n-1})}$ could compensate for this weaker integrability (recall that $d^\ast \psi \in L^{(n,1)}$). However, the stronger assumption becomes essential in \eqref{eq:containing-dvarphi}. There, if we only had $d\xi \in L^n$, we would need $d(\varphi P) \in L^{(n,\frac{n}{n-1})}$, which requires $dP \in L^{(n,\frac{n}{n-1})}$–a condition that not be satisfied in general.
\end{remark}

\vspace*{0.5cm}

\textbf{Estimate containing $\mathfrak{h}$.}

Since $P \in \mathrm{SO}(m)$, we have:
\begin{equation} \label{eq:first-estimate-h}
   \Bigg\vert  \int_{B_r(x)} P^{-1}du \wedge\ast \mathfrak{h} \Bigg\vert \lesssim \Vert du \Vert_{\lorentzdual(B_r(x))} \Vert \mathfrak{h}\Vert_{L^{(n,1)}(B_r(x))}.
\end{equation}
In order to apply interior estimates for harmonic forms, we exploit the fact that we have extended the test function $\phi$ to $\widetilde{\phi}$, and we claim that for  $\delta \in (0,\frac{3}{4})$ we have:
\begin{align}
   \Vert \mathfrak{h} \Vert_{L^{(n,1)}(B_r(x))} \leq C \delta^{n-1}. \label{eq:second-estimate-claim}
\end{align}

\begin{proof}[Proof of the Claim]
By a duality argument, we have:
\begin{equation}
    \Vert \mathfrak{h} \Vert_{L^{(n,1)}(B_r(x))} = \sup_{ \substack{ k \in \lorentzdual (B_r(x)) \\ \Vert k \Vert_{\lorentzdual(B_r(x)) } \leq 1 }} \int_{B_r(x)} \mathfrak{h} \cdot k \; dvol_n. \label{eq:step1-of-claim}
\end{equation}
Consider an arbitrary but fixed $k = df + d^\ast g + \mathfrak{H} \in \lorentzdual(B_r(x))$ with $df \in dW_T^{1,(\frac{n}{n-1},\infty)}(B_r(x))$, $d^\ast g \in d^\ast W_N^{1, (\frac{n}{n-1},\infty)}(B_r(x), \wedge^2)$, and $\mathfrak{H} \in C^\infty(B_r(x),\wedge^1)$ and satisfy 
\begin{equation}
\begin{aligned}
    \Vert df \Vert_{L^{(\frac{n}{n-1},\infty)}(B_r(x))}& + \Vert d^\ast g \Vert_{L^{(\frac{n}{n-1},\infty)}(B_r(x)} + \Vert \mathfrak{H} \Vert_{L^{(\frac{n}{n-1},\infty)}(B_r(x)} \\
    &\leq C \Vert k \Vert_{L^{(\frac{n}{n-1},\infty)}(B_r(x)} 
    \leq C. \label{eq:estimate-by-one-new}
    \end{aligned}
\end{equation}

Since $\widetilde{\phi}$ is supported in $B_{\delta r}(x)$ with $\Vert \widetilde{\phi}\Vert_{L^{(n,1)}(B_r(x))} \leq 1$, orthogonality yields:
\begin{equation} 
\begin{aligned}
    \int_{B_r(x)} \mathfrak{h} \cdot (df + d^\ast g + \mathfrak{H}) \; dvol_n  
    &= \int_{B_r(x)} \mathfrak{h} \cdot \mathfrak{H} \; dvol_n  \\
    &= \int_{B_r(x)} \widetilde{\phi} \cdot \mathfrak{H} \; dvol_n  
    = \int_{B_{\delta r}(x)} \widetilde{\phi} \cdot \mathfrak{H} \; dvol_n \\
    &\leq \Vert \widetilde{\phi} \Vert_{L^{(n,1)}(B_{\delta r}(x))} \Vert \mathfrak{H} \Vert_{\lorentzdual(B_{\delta r}(x))} \\
    &\leq \Vert \mathfrak{H} \Vert_{\lorentzdual(B_{\delta r}(x))}. \label{eq:step2claim}
\end{aligned}
\end{equation}

Since $\mathfrak{H}$ is harmonic, the following function is monotone increasing for any $1 \leq p< \infty$ \footnote{This is a direct consequence of that fact that for $1 \leq p < \infty$, we have $\Delta \vert \mathfrak
H\vert^p \geq 0$ and hence $\vert \mathfrak{H}\vert^p$ is subharmonic. For details see \ref{lm:lemma-of-harmonic}. }
\begin{align}
    r \mapsto \frac{1}{r^n}\int_{B_r(x)} \vert \mathfrak{H} \vert^{p}. \label{eq:monoton}
\end{align}

Following the strategy in \cite[p. 145]{lecture-notes-tristan}, we choose $\delta \in (0,1)$ independent of $r$ such that $\delta r < \frac{3}{4} r$ holds, i.e., $\delta \in (0,\frac{3}{4})$, which yields

\begin{equation} \label{eq:comparable}
\begin{aligned}
    \left( \frac{1}{\delta r}\right)^n \int_{B_{\delta r}(x)} \vert \mathfrak{H} \vert^p 
    &\leq \left( \frac{1}{3r/4}\right)^n \int_{B_{3r/4}(x)} \vert \mathfrak{H} \vert^p  \\[1ex]
    \Leftrightarrow  \int_{B_{\delta r}(x)} \vert \mathfrak{H} \vert^p 
    &\leq \left( \frac{4\delta}{3}\right)^n \int_{B_{3r/4}(x)} \vert \mathfrak{H} \vert^p . 
\end{aligned}
\end{equation}

Because $\mathfrak{H}$ is harmonic, all local $L^p$ norms are comparable by interior estimates and the mean-value formula. By choosing $1 < p < \frac{n}{n-1}$, we have $L^{(\frac{n}{n-1}, \infty)}(B_r(x)) \hookrightarrow L^p(B_r(x))$ on  $B_r(x)$. Using $\Vert \mathfrak{H} \Vert_{\lorentzdual(B_r(x))} \leq \Vert k\Vert_{\lorentzdual (B_r(x))} \leq 1$, we obtain:
\begin{equation} \label{eq:second-estimate-H}
\begin{aligned}
    \Vert \mathfrak{H} \Vert_{\lorentzdual(B_{\delta r}(x))} 
    &\lesssim (\delta r)^{\frac{n(p-1)}{p}-1} \Vert \mathfrak{H} \Vert_{L^{p}(B_{\delta r}(x))} 
    \\
    &\overset{\eqref{eq:comparable}}{\leq} C(n,p) \delta^{\frac{n}{p}} (\delta r)^{\frac{n(p-1)}{p}-1} \Vert \mathfrak{H} \Vert_{L^{p}(B_r(x))} \\
    &\overset{p<\frac{n}{n-1}}{\leq} C(n,p) \delta^{\frac{n}{p}} (\delta r)^{\frac{n(p-1)}{p}-1} \Vert \mathfrak{H} \Vert_{\lorentzdual(B_r(x))} \Vert 1 \Vert_{L^{(\frac{np}{n-p(n-1)},p)} (B_r(x))} \\
    &\le C(n,p) \delta^{\frac{n}{p}} (\delta r)^{\frac{n(p-1)}{p}-1} r^{\frac{n-(n-1)p}{p}} \Vert \mathfrak{H} \Vert_{\lorentzdual(B_r(x))}  \\
    &= C(n,p) \delta^{n-1} \Vert \mathfrak{H} \Vert_{\lorentzdual(B_r(x))}.
\end{aligned}
\end{equation}
where we used Hölder inequality for the parameters (recall $p<\frac{n}{n-1}$)
\begin{equation}
    \frac{1}{p} = \frac{n-1}{n} + \frac{1}{\frac{np}{n-p(n-1)}}, \qquad \frac{1}{p} = \frac{1}{\infty} + \frac{1}{p}.
\end{equation}
Combining \eqref{eq:step1-of-claim}, \eqref{eq:step2claim}, \eqref{eq:second-estimate-H} and taking the supremum over $k \in \lorentzdual(B_r(x))$ completes the proof of the claim.
\end{proof}

Combining \eqref{eq:first-estimate-h}  with \eqref{eq:second-estimate-claim}  then yields
\begin{equation} \label{eq:containing-h}
   \Bigg\vert  \int_{B_r(x)} P^{-1} du \wedge \ast \mathfrak{h} \Bigg\vert \leq C''' \delta^{n-1} \Vert \nabla u\Vert_{\lorentzdual (B_r(x))}. 
\end{equation}

\vspace*{0.5cm}
\textbf{The final bound.} 

Combining the estimates \eqref{eq:containing-dvarphi}, \eqref{eq:containing-dpsi}, and \eqref{eq:containing-h}, taking the supremum over all $\phi \in \lorentzone(B_{\delta r}(x),\wedge^1)$ with $\Vert \phi \Vert_{\lorentzone(B_{\delta r}(x))} \leq 1$, and insert into \eqref{eq:containing-step-by-step}, we end up with:
\begin{align}
    \Vert \nabla u\Vert_{\lorentzdual(B_{\delta r}(x))} \leq \Big( \varepsilon_0 (C'+C'') + \delta^{n-1}C''' \Big) \Vert \nabla u\Vert_{\lorentzdual(B_r(x))}.
\end{align}

Fix an arbitrary exponent $\alpha \in (0, n-1)$. Since $n - 1 - \alpha > 0$, we can choose $\delta \in (0, 3/4)$ small enough such that
\begin{equation} \label{eq:def-of-delta}
    C''' \delta^{n-1} \leq \frac{1}{2} \delta^\alpha \quad \Leftrightarrow  \quad \delta \leq \left( \frac{1}{2C'''} \right)^{\frac{1}{n-1-\alpha}}. 
\end{equation}
With $\delta \in (0, 3/4)$ now fixed, we choose $\varepsilon_0 > 0$ sufficiently small such that
\begin{equation}
    \varepsilon_0 (C' + C'') \leq \frac{1}{2} \delta^\alpha. \label{eq:eps-morrey}
\end{equation}
Combining \eqref{eq:def-of-delta} and \eqref{eq:eps-morrey} yields the following decay:
\begin{equation} \label{eq:one-iteration}
    \Vert \nabla u\Vert_{\lorentzdual(B_{\delta r}(x))} \leq \delta^\alpha \Vert \nabla u\Vert_{\lorentzdual(B_r(x))}. 
\end{equation}

\vspace*{0.5cm}

\textbf{Iteration argument.} 

Let $0 < r < r_0$ be arbitrary. Then there exists $k \in \mathbb{N}_0$ such that 

\begin{equation} \label{eq:iteration}
    \delta^{k+1} r_0 \leq r < \delta^k r_0 \qquad \Leftrightarrow \qquad  \delta^{k+1} \leq \frac{r}{r_0} < \delta^k.  
\end{equation}
Iterating \eqref{eq:one-iteration} $k$ times yield:
\begin{equation} \label{eq:iteration-1}
\begin{aligned}
   \Vert \nabla u\Vert_{\lorentzdual (B_{r }(x))} 
   &\leq \Vert \nabla u\Vert_{\lorentzdual(B_{\delta^{k} r_0}(x))} \leq \delta^{\alpha k} \Vert \nabla u \Vert_{\lorentzdual(B_{r_0}(x)) } \\
   &\leq \delta^{\alpha (k+1)} \frac{1}{\delta^\alpha }\Vert \nabla u \Vert_{\lorentzdual(B_{r_0}(x)) } \\
   &\overset{\eqref{eq:iteration}}{\leq} \left( \frac{r}{r_0} \right)^{\alpha } \frac{1}{\delta^\alpha } \Vert \nabla u \Vert_{\lorentzdual(B_{r_0}(x)) } \\
   &\overset{\eqref{eq:def-of-delta}}{=} \left( \frac{r}{r_0} \right)^{\alpha } (2C''')^{\frac{\alpha}{n-1-\alpha}} \Vert \nabla u \Vert_{\lorentzdual(B_{r_0}(x)) }. 
\end{aligned}
\end{equation}

Taking the supremum over $0 < r < r_0$ and $x \in B_{1-\tau_0}(0)$, we obtain for any $\alpha \in (0, n-1)$ a constant $C(n, \alpha, r_0) > 0$ such that
\begin{align}
    \sup_{\substack{x \in B_{1-\tau_0}(0) \\ 0 < r < r_0}} r^{-\alpha }\Vert \nabla u\Vert_{\lorentzdual (B_{r}(x))} \leq C(n, \alpha, r_0) < \infty, \label{eq:morrey-space-estimate}
\end{align}
where $C(n,\alpha,r_0) \coloneqq (2C''')^{\frac{\alpha}{n-1-\alpha}} r_0^{-\alpha} \Vert \nabla u\Vert_{L^{(\frac{n}{n-1},q')}(B^n)}$ and we used $L^{(\frac{n}{n-1},q')} \subset L^{(\frac{n}{n-1},\infty)}$.

\vspace*{0.5cm}
\textbf{Conclusion: Continuity via Morrey space techniques.}  \label{step:conclusion-morrey}

For any $x \in B_{1-\tau_0}(0)$ and $0 < r < r_0$, Hölder's inequality together with \eqref{eq:morrey-space-estimate} yields:
\begin{equation}
\begin{aligned}
    \Vert \nabla u\Vert_{L^1(B_r(x))} 
    &\lesssim \Vert \nabla u \Vert_{L^{(\frac{n}{n-1},\infty)}(B_r(x))} \Vert 1 \Vert_{L^{(n,1)}(B_r(x))} \\
    &\leq \tilde{C} r^\alpha \cdot r = \tilde{C} r^{\alpha+1}, \label{eq:estimate-morrey}
\end{aligned}
\end{equation}
where we used that $\Vert 1 \Vert_{L^{(n,1)}(B_r(x))}\lesssim r$. 

Since \eqref{eq:estimate-morrey} holds for any $\alpha < n-1$, choosing $n-2 < \alpha < n-1$ and defining $\mu \coloneqq \alpha - n + 2 \in (0,1)$, the decay estimate can be rewritten as:
\begin{equation}
    \Vert \nabla u\Vert_{L^1(B_r(x))} \leq \tilde{C} r^{n-1+\mu}.\label{eq:morrey-growth-final}
\end{equation}

By Morrey's Dirichlet Growth Theorem \cite[Thm. 3.5.2]{morrey}, \eqref{eq:morrey-growth-final} implies Hölder continuity on the domain $B_{1-\tau_0}(0)$, i.e., $u \in C^{0,\mu}(B_{1-\tau_0}(0))$. 

Since $\tau_0 \in (0,1)$ was arbitrary, we conclude that $u \in C^{0,\mu}_{\text{loc}}(B^n)$, completing the proof of Theorem \ref{thm:main-thm}.

\end{proof}

\vspace{0.7cm}

\section{Proof of Theorem \ref{thm:main-theorem-2} –  a "full" Conservation law.} \label{sec:version2}

The entire section follows the strategy of Rivière in \cite[Thm. I.1, I.3, I.4]{conservation-law-conf-invariant}.

\begin{theorem} \label{thm:existence-a-b}
      There exists $\varepsilon_0>0$ and $C_0>0$ such that for any $\Omega \in  L^{(n,q)}(B^n, \mathfrak{so}(m)\otimes \wedge^1)$  satisfying 
  \begin{align}
      \Vert \Omega \Vert_{\lorentzl} \leq C \varepsilon_0, \label{eq:smallness-a-b-gauge}
  \end{align}
    then there exists $A \in L^\infty \cap W^{1,(n,q)}(B^n,GL(m))$ and $B \in W^{1,(n,q)}(B^n, \R^{m\times m} \otimes \wedge^{2})$ satisfying 
    \begin{equation} \begin{gathered}
    \begin{cases}
  \Vert  \operatorname{dist}(A, \mathrm{SO}(m)) \Vert_{L^\infty(B^n)}  \leq C_0 \Vert \Omega \Vert_{\lorentzl (B^n)}, \\  \\ 
       \Vert d^\ast B \Vert_{L^{(n,q)}}  \leq C_0 \Vert \Omega \Vert_{L^{(n,q)}}, \\ \\
      -d^\ast B = dA -A\Omega. \label{eq:apriori-a-b}
      \end{cases}
    \end{gathered} \end{equation} 
\end{theorem}

\begin{proof}
    Since \eqref{eq:smallness-a-b-gauge} is satisfied, we  apply again Lemma \ref{lm:constructing-gauge-p-original-space} and we follow the existence of $P \in W^{1,(n,q)}(B^n, \mathrm{SO}(m))$, $\xi \in W^{1,(n,q)}(B^n, \mathfrak{so}(m)) \otimes \wedge^{n-2})$ which satisfies the system \eqref{eq:gauge-system} as well as the a-priori bound:
    \begin{equation}
        \begin{aligned}
            \Vert d\xi \Vert_{L^{(n,q)}(B^n)} + \Vert dP \Vert_{L^{(n,q)}(B^n)} \leq C\Vert \Omega\Vert_{\lorentzl(B^n)}. \label{eq:xi-p-for-a-b}
        \end{aligned}
    \end{equation}
    We make the ansatz: 
    \begin{equation}
        \begin{aligned}
            A\coloneqq (\mathrm{Id}_m+\sigma)  P^{-1} \qquad \Rightarrow 
            dA=d \sigma P^{-1} + (I+ \sigma)dP^{-1}, \label{eq:polar-decompositon-ansatz}
        \end{aligned}
    \end{equation}
   for some $\sigma \in L^\infty \cap W^{1,(n,1)}(B^n, \R^{m\times m})$ with $\Vert \sigma \Vert_{L^\infty} <1$. A direct consequence of the assumption is that $A$ is invertible due to Neumann series. 
   This assumption will be justified later, once the a-priori bounds \eqref{eq:apriori-a-b} have been established. 

  We follow the strategy of Rivière  in \cite[Thm. I.4]{conservation-law-conf-invariant}. (Note that he constructs $(\mathrm{Id}_m+\sigma)$  by finding $\tilde{A}$ and $\sigma$ takes the role of $\hat{A}$ there.)
   \vspace{0.5cm}

We seek for $A$ and $B$ satisfying \eqref{eq:apriori-a-b}. Using ansatz \eqref{eq:polar-decompositon-ansatz} then yields
\begin{equation}
\begin{aligned}
   -d^\ast B= dA -A \Omega , \quad 
   &\Leftrightarrow \quad  -d^\ast B \overset{\eqref{eq:polar-decompositon-ansatz}}{=} d\sigma P^{-1} +(\mathrm{Id}_m+\sigma)dP^{-1} - (\mathrm{Id}_m + \sigma)P^{-1} \Omega \\
   & \Leftrightarrow \quad  -d^\ast B P =d\sigma  -(\mathrm{Id}_m+\sigma) \underbrace{( -dP^{-1} P +P^{-1} \Omega P )}_{=\ast d\xi }  \\
   & \Leftrightarrow \quad  -d^\ast B P =d\sigma  -(\mathrm{Id}_m+\sigma) \ast d\xi. \label{eq:equivalent-db}
\end{aligned}
\end{equation}
Applying the $d^\ast (\cdot )$-operator leads to  
\begin{equation}
    \begin{aligned}
        -\Delta \sigma + (-1)^{n-1} \ast (d\sigma \wedge d\xi) = d^\ast B \cdot dP. \label{eq:pde-for-sigma}
    \end{aligned}
\end{equation}

Note that \eqref{eq:equivalent-db} is multiplying again by $P^{-1}$ equivalent to
\begin{align}
    (I+\sigma) \ast d\xi P^{-1}= d^\ast B + d\sigma P^{-1},
\end{align}
and hence by applying the $d(\cdot)$ operator (and recalling that $d(d\sigma P^{-1}) = -d\sigma \wedge dP^{-1}= -d(\sigma dP^{-1})$) this leads to the following system
\begin{align}
    d(\ast d\xi P^{-1}) + d(\sigma \ast d\xi P^{-1})  = \Delta B -d(\sigma dP^{-1}). \label{eq:pde-for-b}
\end{align}

Combining \eqref{eq:pde-for-sigma} and \eqref{eq:pde-for-b},  we  then look for $\sigma \in W^{1,(n,q)}(B^n,\R^{m\times m})$ and $B\in W^{1,(n,q)}(B^n, \R^{m\times m}, \wedge^2)$ which solves 
\begin{equation}
    \begin{aligned}
    \begin{cases}
          -\Delta \sigma =(-1)^n \ast (d\sigma \wedge d\xi) + d^\ast B \cdot dP & \text{ in }B^n, \\
       \Delta B  =   d(\ast d\xi P^{-1}) + d(\sigma \ast d\xi P^{-1})   +d(\sigma dP^{-1})  & \text{ in }B^n, \\
       \partial_\nu \sigma =0, \qquad B=0 & \text{ on } \partial B^n, \\
      \displaystyle \int_{B^n} \sigma = \mathrm{Id}_m.
    \end{cases} \label{eq:system-to-solve-fixpoint}
    \end{aligned}
\end{equation}

For the sake of simplicity, we omit the domain $B^n$. 
Standard elliptic estimates $W^{1,(\frac{n}{2},1)}\hookrightarrow L^{(n,1)}$ (see again Theorem \ref{cor:improved-sobolev}), and Hölder's inequality gives: 
\begin{equation}
    \begin{aligned}
        \Vert d\sigma \Vert_{L^{(n,1)}} 
        &\lesssim  \Vert \Delta \sigma \Vert_{L^{(\frac{n}{2},1)}} \lesssim \Vert d\sigma \Vert_{L^{(n,q')}} \Vert d\xi \Vert_{L^{(n,q)}} +  \Vert d^\ast  B \Vert_{L^{(n,q')}} \Vert dP \Vert_{\lorentzl} \\
        & \overset{\eqref{eq:xi-p-for-a-b}}{\leq}  C' \Vert \Omega \Vert_{\lorentzl}  \Vert d\sigma \Vert_{L^{(n,q')}}  + C' \Vert \Omega \Vert_{\lorentzl}\Vert d^\ast B\Vert_{L^{(n,q')}} ,
    \end{aligned}
\end{equation}
together with $W^{1,(n,1)}\hookrightarrow L^\infty$ and $L^{(n,q)}\subset L^{(n,q')}$ due to $2\leq q \leq q'$, we have
\begin{equation}
    \begin{aligned}
        \Vert d^\ast B & \Vert_{L^{(n,q')}}  \lesssim   \Vert d^\ast B  \Vert_{L^{(n,q)}}
       \lesssim \Vert   d\xi P^{-1} \Vert_{\lorentzl} + \Vert \sigma \ast d\xi P^{-1} \Vert_{\lorentzl}  + \Vert  \sigma  dP^{-1} \Vert_{\lorentzl} \\
        &\quad  \lesssim \Vert   d\xi  \Vert_{\lorentzl} \Vert P^{-1} \Vert_{L^\infty} + \Vert \sigma  \Vert_{L^\infty}\Vert d\xi  \Vert_{\lorentzl} \Vert P^{-1} \Vert_{L^\infty}  + \Vert  \sigma \Vert_{L^\infty} \Vert  dP^{-1} \Vert_{\lorentzl} \\
        & \quad \leq C'' \Vert \Omega \Vert_{\lorentzl} + C''\Vert d\sigma \Vert_{L^{(n,1)}} \Vert \Omega \Vert_{\lorentzl}.
    \end{aligned}
\end{equation}

Choosing $\varepsilon_0$ small enough, by a fix-point argument we obtain the existence of $B$ and $\sigma$, and hence also $A$,  such that
\begin{equation}
    \begin{aligned}
        \Vert d\sigma \Vert_{L^{(n,1)}} + \Vert \sigma \Vert_{L^\infty } + \Vert d^\ast B \Vert_{\lorentzl } \leq C_0 \Vert \Omega \Vert_{\lorentzl}. \label{eq:equation-for-sigma-1}
    \end{aligned}
\end{equation}
Using \eqref{eq:polar-decompositon-ansatz}, $\sigma =AP-\mathrm{Id}_m$, estimate \eqref{eq:equation-for-sigma-1} then leads to
\begin{equation}
    \begin{aligned}
        \Vert AP-\mathrm{Id}_m \Vert_{L^\infty(B^n)} + \Vert d^\ast B \Vert_{\lorentzl (B^n) } \leq  C_0 \Vert \Omega \Vert_{\lorentzl (B^n)}.\label{eq:a-priori-A-final-bound}
    \end{aligned}
    \end{equation}
    Again using the ansatz $A=(\mathrm{Id}_m+\sigma)P^{-1}$, estimate \eqref{eq:equation-for-sigma-1} further leads to
    \begin{equation}
    \begin{aligned}
        \Vert dA \Vert_{L^{(n,q)}} &\lesssim \Vert d\sigma P^{-1} \Vert_{L^{(n,q)}}+ (1+ \Vert  \sigma \Vert_{L^\infty}) \Vert dP^{-1} \Vert_{L^{(n,2)}} \\
         &\overset{1 \leq q}{\lesssim} \Vert d\sigma \Vert_{L^{(n,1)}}+ C'\Vert \Omega\Vert_{L^{(n,q)}} + C'' \Vert \Omega\Vert_{L^{(n,q)}}^2 \\
        &\lesssim C \Vert \Omega \Vert_{L^{(n,q)}}
    \end{aligned}
    \end{equation}
    for a small enough $\varepsilon_0>0$.
    It remains to discuss the bound for $dA^{-1}$. We have $AA^{-1}= \mathrm{Id}_m$ and hence 
    \begin{equation}
        \begin{aligned}
            dA A^{-1 }+ AdA^{-1}=0 \quad \Rightarrow \quad dA^{-1}= - A^{-1}dAA^{-1}. \label{eq:a-minus-one}
        \end{aligned}
    \end{equation}
    Equation \eqref{eq:a-minus-one} leads to 
    \begin{equation}
        \begin{aligned}
            \Vert dA^{-1} \Vert_{\lorentzl} \lesssim \Vert A^{-1} \Vert_{L^\infty}^2 \Vert dA \Vert_{\lorentzl}.  
        \end{aligned}
    \end{equation}
    If we show that $A^{-1}$ is  bounded in $L^\infty$, we can conclude by using \eqref{eq:a-priori-A-final-bound}. Since the Neumann series is well-defined by assumption, we have:
    \begin{align}
        A^{-1}= \big( (\mathrm{Id}_m+\sigma)P^{-1})^{-1} = P\sum _{k=0}^\infty (-\sigma )^{k}. \label{eq:neumanseries}
    \end{align}
Choosing $\varepsilon_0<\frac{1}{2C_0}$, \eqref{eq:neumanseries} is absolute-convergent and we get
\begin{align}
    \sum_{k=0}^\infty \Vert \sigma \Vert_{L^\infty}^{k} = \frac{1}{1-\Vert \sigma \Vert_{L^\infty}} \leq \frac{1}{1-  C_0\varepsilon_0}  \overset{\text{choice of }\varepsilon_0}{\leq} 2 <\infty, \label{eq:a-inverse-bounded}
\end{align}

\vspace{0.5cm}

It remains to show that the equality $-d^\ast B= dA-A\Omega$ holds.

\vspace{0.5cm}
Since $B$ is a solution of \eqref{eq:system-to-solve-fixpoint}, $(\mathrm{Id}_m+\sigma)\ast d\xi P^{-1}- d^\ast B - d\sigma P^{-1 }$ is exact, and hence there exists a $C \in \lorentzw(B^n)$ such that
    \begin{equation}
        \begin{aligned}
           dC &= (\mathrm{Id}_m+\sigma)\ast d\xi P^{-1}- d^\ast B - d\sigma P^{-1 }, \\
           \Leftrightarrow \quad   dCP &= (\mathrm{Id}_m+\sigma)\ast d\xi - d^\ast BP - d\sigma.
        \end{aligned} \label{eq:introduction-c}
    \end{equation}
Applying the $d^\ast(\cdot)$-operator on \eqref{eq:introduction-c}, one gets 
\begin{equation}
    \begin{aligned}
        \Delta C P + dC \cdot dP &= \ast ( d\sigma \wedge  d\xi )- d^\ast B \cdot dP - \Delta \sigma \overset{\eqref{eq:system-to-solve-fixpoint}}{=} 0. \label{eq:c-solves-intro}
    \end{aligned}
\end{equation}
 We want to show that $C=0$, from which \eqref{eq:system-to-solve-fixpoint} follows immediately.

Applying $P^{-1}$ on both sides of \eqref{eq:c-solves-intro} and using $dPP^{-1} = -P dP^{-1}$, $C$ solves
\begin{equation}
    \begin{aligned}
        \begin{cases}
            \Delta C = dC \cdot P dP^{-1} & \text{ in } B^n, \\
            C=0 & \text{ on } \partial B^n.
        \end{cases} \label{eq:pde-for-c}
    \end{aligned}
\end{equation}
Using once more the Calderón-Zygmund estimates, we get
\begin{equation}
    \begin{aligned}
        \Vert d C \Vert_{L^{(n,q)}} 
        &\overset{1\leq q}{\lesssim} \Vert d C \Vert_{L^{(n,1)}} \lesssim \Vert \Delta C  \Vert_{L^{(\frac{n}{2},1)}}  \lesssim \Vert dC \Vert_{L^{(n,q)}} \Vert P \Vert_{L^\infty} \Vert dP \Vert_{L^{(n,q')}} \\
        &\lesssim  \varepsilon_0 \Vert d C \Vert_{L^{(n,q)}},
    \end{aligned}
\end{equation}
and hence $\nabla C \equiv 0$. Due to the zero boundary conditions in \eqref{eq:pde-for-c}, we conclude $C=0$.
This finishes the proof of Theorem \ref{thm:existence-a-b}. 

\end{proof}

We now give the continuity proof of solutions $u\in W^{1,(\frac{n}{n-1},q')}(B^n)$ to system \eqref{eq: model-pde}. 

\vspace{0.5cm}

\begin{proof}[Proof of Theorem \ref{thm:main-theorem-2}.]
We argue analogously as in the proof of Section \ref{sec:version1}.

Since continuity is a local property, it suffices to prove continuity on the ball $B_{1-\tau_0}(0)$ for an arbitrary $\tau_0 \in (0,1)$. Since the $L^{(n,q)}$-norm is absolutely continuous, we can find a uniform radius $R_0 > 0$ such that for any $R < R_0$ and any $x \in B_{1-\tau_0}(0)$, we have $B_R(x) \subset B^n$ and $\Omega \in L^{(n,q)}(B^n, \mathfrak{so}(m) \otimes \wedge^1)$ satisfies the smallness threshold given in Theorem \ref{thm:existence-a-b}:
\begin{align}
    \Vert \Omega \Vert_{\lorentzl (B_R(x))}\leq C\varepsilon_0. \label{eq:assumption-ball-a-b}
\end{align}
During the proof, it is possible to adopt $\varepsilon_0$ as small as is necessary.
Theorem \ref{thm:existence-a-b} yields the existence of $A \in L^\infty \cap \lorentzw(B_R(x),\mathrm{GL}(m))$ and $B \in \lorentzw(B_R(x), \mathfrak{so}(m)\otimes \wedge^{2})$ such that 
\begin{align}
    -d^\ast B = dA -A \Omega \qquad \text{ in }B_R(x).  
\end{align}
The conservation law claimed in  \eqref{eq:conservation-law} then follows by a direct computation: 
     \begin{equation} \label{eq:conservation-law-ab}
        \begin{aligned}
            d(\ast A du ) &= dA \wedge \ast du + A \ast \Delta u 
            = dA\wedge \ast du - A\Omega \wedge \ast du \\
            &\overset{\eqref{eq:def-b-star}}{=} -d^\ast B \wedge \ast du 
            = -(-1)^{n+1} \ast d\ast B \wedge \ast du \\
            &= (-1)^n \ast d\ast B \wedge \ast du 
            = -(-1)^{n-1} d\ast B \wedge du \\
            &= - (-1)^{n-1}d( (\ast B)\wedge du ).
        \end{aligned}
    \end{equation}

Next, by a Hodge decomposition there exist $d\alpha \in dW_T^{1,(\frac{n}{n-1},2)}(B_R(x))$, 
$d^\ast \beta \in d^\ast W_N^{1,(\frac{n}{n-1},2)}(B_R(x) ,  \wedge^{n-2})$, and $\mathfrak{h}\in C^\infty(B_R(x),\wedge ^1)$ such that
\begin{equation} \begin{aligned}
\begin{cases}
   A du= d \alpha + \ast d \beta + \mathfrak{h} &\text{ in }B_R(x), \\
   \alpha =\ast\beta=0 &\text{ on } \partial B_R(x),\\
   \mathfrak{h}=Adu &\text{ on } \partial B_R(x).
   \end{cases} \label{eq:hodge-decomps}
\end{aligned} \end{equation}
Using \eqref{eq:conservation-law-ab}, $\alpha$ and $\beta$ then solve:
\begin{equation} \begin{aligned}
\begin{cases}
   - \Delta \alpha = d^\ast (A du) = d^\ast B \cdot du & \text{ in } B_R(x), \\
    \alpha = 0 & \text{ on } \partial B_R(x),
    \end{cases} \label{eq:system-alpha}
\end{aligned} \end{equation}
and 
\begin{equation} \begin{aligned}
\begin{cases}
    \Delta \beta = \ast d (A du) = \ast (dA \wedge du )& \text{ in } B_R(x), \\
    \ast \beta = 0 & \text{ on } \partial B_R(x).
    \end{cases} \label{eq:system-beta}
\end{aligned} \end{equation}

Using again the endpoint interpolation of Riesz potentials together with $L^{(\frac{n}{n-1},1)}\hookrightarrow L^{(\frac{n}{n-1},q')}$ ($q' \geq 2$), we get the existence of a constant $C>0$ such that
\begin{equation}
    \begin{aligned}
        \Vert d\alpha \Vert_{L^{(\frac{n}{n-1},q')}(B_R(x))} &\lesssim \Vert d^\ast B \Vert_{\lorentzl (B_R(x)) } \Vert du \Vert_{L^{(\frac{n}{n-1},q')}(B_R(x))} \\
        &\leq C \varepsilon_0 \Vert du \Vert_{L^{(\frac{n}{n-1},q')}(B_R(x))}, \\
        \Vert d^\ast \beta \Vert_{L^{(\frac{n}{n-1},q')}(B_R(x))} &\lesssim \Vert dA\Vert_{\lorentzl (B_R(x)) } \Vert du \Vert_{L^{(\frac{n}{n-1},q')}(B_R(x))} \\
        &\leq C \varepsilon_0 \Vert du \Vert_{L^{(\frac{n}{n-1},q')}(B_R(x))}.
    \end{aligned}
\end{equation}

\textbf{Morrey space decay estimate.}

Let $\delta \in (0,\frac{3}{4})$. We follow the strategy of the proof in Section \ref{sec:version1}. Using the fact that $A\in\mathrm{GL}(m)$, \eqref{eq:hodge-decomps} can be rewritten as $du = A^{-1} d\alpha + A^{-1}\ast d \beta + A^{-1}\mathfrak{h}$, we get:

\begin{equation} \begin{aligned}
    \Vert &\nabla u \Vert_{L^{(\frac{n}{n-1},q')}(B_{\delta R}(x))}  \\
    &\lesssim \Vert A^{-1} \Vert_{L^\infty(B_{\delta R})} \Big( \Vert d\alpha \Vert_{L^{(\frac{n}{n-1},q')}(B_{\delta R}(x))} + \Vert d^\ast \beta \Vert_{L^{(\frac{n}{n-1},q')}(B_{\delta R}(x))} + \Vert \mathfrak{h} \Vert_{L^{(\frac{n}{n-1},q')}(B_{\delta R}(x))} \Big) \\
    &\overset{(\star )}{\lesssim} \Vert d\alpha \Vert_{L^{(\frac{n}{n-1},q')}(B_R(x))} + \Vert d^\ast \beta \Vert_{L^{(\frac{n}{n-1},q')}(B_R(x))} + \delta^{n-1}\Vert \mathfrak{h} \Vert_{L^{(\frac{n}{n-1},q')}(B_R(x))} \\
    & \overset{\eqref{eq:hodge-decomps}}{\lesssim} \Vert d\alpha \Vert_{L^{(\frac{n}{n-1},q')}(B_R(x))} + \Vert d^\ast \beta \Vert_{L^{(\frac{n}{n-1},q')}(B_R(x))} \\
    &\qquad \; + \delta^{n-1} \Big(\Vert A du \Vert_{L^{(\frac{n}{n-1},q')}(B_R(x))} + \Vert d\alpha \Vert_{L^{(\frac{n}{n-1},q')}(B_R(x))} + \Vert d^\ast \beta \Vert_{L^{(\frac{n}{n-1},q')}(B_R(x))} \Big) \\
    &\lesssim \Bigg( \varepsilon_0 + \delta^{n-1} + \Vert A \Vert_{L^\infty} \delta^{n-1} \Bigg) \Vert \nabla u \Vert_{L^{(\frac{n}{n-1},q')}(B_R(x))} \\
    &\leq C' \Bigg( \delta^{n-1}+ \varepsilon_0 \Bigg)\Vert \nabla u \Vert_{L^{(\frac{n}{n-1},q')}(B_R(x))}.
\end{aligned} \end{equation}
The estimate in ($\star$) needs to be discussed. Again, we use the monotonicity properties of harmonic functions/forms and refer the reader to the Proof in Section \ref{sec:version1}, specifically equation \eqref{eq:comparable}. In that argument, one may choose $p=\frac{n}{n-1}$, since the exponent does not need to be modified again on a larger ball and no further interior estimates are required. 

Note that due to the construction of $A$ in Theorem \ref{thm:existence-a-b}, $A^{-1}$ is bounded for the choice of $\varepsilon_0<\frac{1}{2C_0}$ ( see \eqref{eq:a-inverse-bounded}): 
\begin{equation} \begin{aligned}
    \Vert A^{-1}\Vert_{L^\infty} = \Vert P (\mathrm{Id}_m+ \sigma)^{-1} \Vert_{L^\infty} \lesssim \frac{1}{1-\varepsilon_0}\leq 2. 
\end{aligned} \end{equation}

Fix an arbitrary exponent $\alpha \in (0, n-1)$. Since $n - 1 - \alpha > 0$, we can choose $\delta \in (0, 3/4)$ small enough such that
\begin{equation} \label{eq:def-of-delta-a-b}
  C' \delta^{n-1} \leq \frac{1}{2} \delta^\alpha \quad \Leftrightarrow \quad \delta \leq \left( \frac{1}{2C'} \right)^{\frac{1}{n-1-\alpha}}. 
\end{equation}
With $\delta \in (0, 3/4)$ now fixed, we choose $\varepsilon_0 > 0$ sufficiently small such that
\begin{equation}
   C' \varepsilon_0 \leq \frac{1}{2} \delta^\alpha. \label{eq:eps-morrey-a-b}
\end{equation}

\vspace{0.5cm}

\textbf{Iteration argument.}

Following the iteration argument in Section \ref{sec:version1}, for any $\alpha \in (0, n-1)$ and constant $C(n, \alpha, R_0) > 0$, we end up with:
\begin{align}
    \sup_{\substack{x \in B_{1-\tau_0}(0) \\ 0 < R < R_0}} R^{-\alpha }\Vert \nabla u\Vert_{L^{(\frac{n}{n-1},q')} (B_{R}(x))} \leq C(n, \alpha, R_0) < \infty, \label{eq:morrey-space-estimate-a-b}
\end{align}
where $C(n,\alpha,R_0) \coloneqq (2C')^{\frac{\alpha}{n-1-\alpha}} R_0^{-\alpha} \Vert \nabla u\Vert_{L^{(\frac{n}{n-1},q')}(B^n)}$ and we used $B_R(x) \subset B^n$.

\vspace*{0.5cm}

\textbf{Conclusion: Higher integrability of $\nabla u$ via Morrey-Adams techniques.} 

Let $x \in B_{1-\tau_0}(0)$, $R<R_0$, and $\alpha \in (0,n-1)$. Equation \eqref{eq:morrey-space-estimate-a-b} leads to 
\begin{equation}
    \begin{aligned}
        \int_{B_R(x)} |\Delta u| \lesssim \Vert \Omega \Vert_{L^{(n,q)}(B_R(x))} \Vert \nabla u\Vert_{L^{(\frac{n}{n-1},q')}(B_R(x))} \lesssim \Vert \Omega\Vert_{L^{(n,q)}(B_R(x))} R^{\alpha}.
    \end{aligned}
\end{equation}
Standard estimates by Adams \cite[Prop. 3.2.(ii)]{riesz-adams} then yield $\nabla u \in L_{\mathrm{loc}}^s$ and hence $\nabla u \in L_{\mathrm{loc}}^{(\bar{s},q')}$ for $s>\bar{s}>\frac{n}{n-1}$ due to $L^s \subset L^{(\bar{s},q')}$. 
Indeed, by weak fractional integration, we conclude by Adams
\begin{equation}
    \begin{aligned}
      \Vert \nabla u \Vert _{L^{s}(B_{R/2}(x))} \lesssim \Vert I_1 (|\Delta u|) \Vert_{L^{s}(B_R(x))} < \infty.
    \end{aligned}
\end{equation}
We have thus improved the integrability of the gradient to $\nabla u\in L_{\mathrm{loc}}^{(\bar{s},q')}(B_R(x))$ with $s>\bar{s}>\frac{n}{n-1}$. Consequently, the right-hand side of $-\Delta u = \Omega \cdot \nabla u$ becomes subcritical via Hölder's inequality on Lorentz spaces, $L_{\mathrm{loc}}^{(n,q)} \cdot L_{\mathrm{loc}}^{(\bar{s},q')} \hookrightarrow L_{\mathrm{loc}}^r$ for some $r>1$. By standard Calderón-Zygmund estimates and  bootstrapping arguments, one obtains
\begin{equation}
    \begin{aligned}
        \Delta u \in L_{\mathrm{loc}}^{p} \qquad \text{for every } n>p>\frac{n}{2}.
    \end{aligned}
\end{equation}
Note that the regularity of $\Omega \in L^{(n,q)}$ limits the regularity of $\Delta u$. Indeed, we can bootstrap $\nabla u \in L^s$ for any $s<\infty$ arbitrary large. However, due to the structure of the equation, we have that $p$ is defined by 
\begin{equation}
    \frac{1}{s}+ \frac{1}{n}= \frac{1}{p} \qquad \Leftrightarrow \qquad p =  \frac{ns}{n+p}.
\end{equation}
$p=p(s)$ is monotone increasing in $s$ and hence bounded by $p=p(s) \to n$ for $s\to \infty$. 
Finally, by the Sobolev embedding $W_{\mathrm{loc}}^{2,{p}} (B^n)\hookrightarrow C_{\mathrm{loc}}^{0,\beta}(B^n)$ with $\frac{n}{2}<p<n$ and $\beta \coloneqq 2-\frac{n}{p}$ we obtain Hölder continuity, which finishes the proof. 

\end{proof}

\section{The Optimality of the Result.} \label{sec:counterexample}

\begin{proposition}
Let $n \ge 3$. For every $\varepsilon_0 > 0$, there exist $u \in W^{1\frac{n}{n-1}}(B^n, \mathbb{R}^{n+1})$ and $\Omega \in L^n(B^n, \mathfrak{so}(n+1) \otimes \wedge^1 )$ such that
\begin{equation}
\|\Omega\|_{L^n(B^n)} < \varepsilon_0, \quad -\Delta u = \Omega \cdot \nabla u \quad \text{in } \mathcal{D}'(B^n), \quad 
\end{equation}
but $u$ is unbounded at the origin and hence $u$ has no continuous representative.
\end{proposition}

\begin{proof}
We divide the construction and verification into several steps.
\vspace{0.5cm}

\textbf{Step 1. Polar coordinates and elementary computations. } 

Write
\begin{equation}
r := |x| = \sqrt{\sum_{i=1}^n x_i^2}, \quad \theta := \frac{x}{|x|} \in S^{n-1}, \quad t := T + \log \frac{1}{r} = T - \log r, \quad  \label{eq:definition-polar}
\end{equation}
where $T \ge 2$ will be chosen at the end. Differentiating $t$ with respect to $r$ gives $\frac{dt}{dr} = -\frac{1}{r}$, so $dt = -\frac{dr}{r}$.
For $k \in \{1, \dots, n\}$, the functions $\theta_k = x_k/r$ satisfy the following.

\vspace{0.5cm}

Orthogonality $\langle d\theta_k, dr \rangle = 0$:
First, differentiating $r^2 = \sum_{i=1}^n x_i^2$  yields 
\begin{equation}
    2r \, dr = \sum_{i=1}^n 2x_i \, dx_i,
\end{equation}
which leads to
\begin{equation}
dr = \sum_{i=1}^n \frac{x_i}{r} dx_i = \sum_{i=1}^n \theta_i dx_i.
\end{equation}
Since  $\langle dx_i, dx_j \rangle = \delta_{ij}$, we get
\begin{equation}
\langle dx_k, dr \rangle = \left\langle dx_k, \sum_{i=1}^n \theta_i dx_i \right\rangle = \theta_k, \quad \text{and} \quad |dr|^2 = \langle dr, dr \rangle = \sum_{i=1}^n \theta_i^2 = 1.
\end{equation}
Because $d\theta_k = d\left(\frac{x_k}{r}\right) = \frac{dx_k - \theta_k dr}{r}$, we get
\begin{equation}
\langle d\theta_k, dr \rangle = \frac{\langle dx_k, dr \rangle - \theta_k |dr|^2}{r} = \frac{\theta_k - \theta_k(1)}{r} = 0.
\end{equation}

 \vspace{0.5cm}
 
 Further, we have
    \begin{equation} \label{eq:equal-one}
    \sum_{k=1}^n \theta_k^2 = \sum_{k=1}^n \frac{x_k^2}{r^2} = \frac{r^2}{r^2} = 1.
    \end{equation}

\vspace{0.5cm}

 The norm of the gradient satisfies $\displaystyle \sum_{k=1}^n |d\theta_k|^2 = \frac{n-1}{r^2}$.
    Using $|dx_k|^2 = 1$ and $\langle dx_k, dr \rangle = \theta_k$:
    \begin{equation}
    |d\theta_k|^2 = \frac{|dx_k|^2 - 2\theta_k \langle dx_k, dr \rangle + \theta_k^2 |dr|^2}{r^2} = \frac{1 - 2\theta_k^2 + \theta_k^2}{r^2} = \frac{1 - \theta_k^2}{r^2}.
    \end{equation}
    Summing over $k=1, \dots, n$:
    \begin{equation}\label{eq:abs-value-theta-k}
    \sum_{k=1}^n |d\theta_k|^2 = \frac{\sum_{k=1}^n 1 - \sum_{k=1}^n \theta_k^2}{r^2} = \frac{n - 1}{r^2}. \quad 
    \end{equation}

Let $a = a(t)$ be a smooth function, and let $a'$ denote the differentiation with respect to $t$. Using the chain rule with $\partial_r t = -1/r$:
\begin{equation}
\partial_r a = a'(t) \cdot \partial_r t = -\frac{a'}{r}. \quad 
\end{equation}
Differentiating again with respect to $r$ leads to:
\begin{equation}
\partial_{rr} a = \partial_r \left( -\frac{a'}{r} \right) = \frac{a'}{r^2} - \frac{\partial_r (a')}{r} = \frac{a'}{r^2} - \frac{a'' \cdot (-1/r)}{r} = \frac{a'' + a'}{r^2}. \quad 
\end{equation}
Consequently, the Laplacian in spherical coordinates is given by
\begin{equation}\label{eq:2}
\Delta a = \partial_{rr} a + \frac{n-1}{r} \partial_r a = \frac{a'' + a'}{r^2} + \frac{n-1}{r}\left(-\frac{a'}{r}\right) = \frac{1}{r^2} \left[ a'' - (n-2)a' \right]. \quad 
\end{equation}

Moreover, we have:
\begin{equation}\label{eq:3}
d(a\theta_k) = a d\theta_k + \theta_k da = a d\theta_k + \theta_k a' dt = a d\theta_k - a' \theta_k \frac{dr}{r}. \quad 
\end{equation}

Using $\Delta_{S^{n-1}} \theta_k = -(n-1)\theta_k$, and the fact that $\nabla a$ is purely radial while $\nabla \theta_k$ is tangential (so $\langle \nabla a, \nabla \theta_k \rangle = 0$), we get
\begin{equation}\label{eq:4}
\Delta(a\theta_k) = (\Delta a)\theta_k + a (\Delta \theta_k) = \frac{1}{r^2} \left[ a'' - (n-2)a' - (n-1)a \right] \theta_k. \quad 
\end{equation}
\vspace{0.5cm}

\textbf{Step 2: Definition of $u$.}

Set
\begin{equation}
\lambda := n - 2 \quad  \label{eq:deif-lambda}
\end{equation}
and define, for $t \ge T$,
\begin{equation}\label{eq:5}
h(t) := \frac{e^{\lambda(t-T)}}{t}, \quad g(t) := \frac{e^{\lambda(t-T)}}{t^{3/2}}. \quad 
\end{equation}
Because $e^{t-T} = e^{\log(1/r)} = r^{-1}$ (see \eqref{eq:definition-polar}), these functions can be written equivalently as
\begin{equation}\label{eq:definition-h-g}
h(t) = \frac{r^{2-n}}{t}, \quad g(t) = \frac{r^{2-n}}{t^{3/2}}. \quad 
\end{equation}
Define for $x\neq 0$
\begin{equation}\label{eq:6}
u(x) := (h(t), g(t)\theta_1, \dots, g(t)\theta_n) \in \mathbb{R} \oplus \mathbb{R}^n = \mathbb{R}^{n+1} \quad 
\end{equation}
and assign an arbitrary value to $u(0)$.

\vspace{0.5cm}

We explicitly compute the derivatives of $h(t)$ and $g(t)$:
\begin{align}
h'(t) &= \frac{\lambda e^{\lambda(t-T)} t - e^{\lambda(t-T)}}{t^2} = h(t) \left( \lambda - \frac{1}{t} \right), \label{eq:7_sub1} \\
h''(t) &= h'(t)\left(\lambda - \frac{1}{t}\right) + h(t)\left(\frac{1}{t^2}\right) = h(t) \left( \lambda^2 - \frac{2\lambda}{t} + \frac{2}{t^2} \right), \label{eq:7_sub2} \\
g'(t) &= \frac{\lambda e^{\lambda(t-T)} t^{3/2} - \frac{3}{2}t^{1/2} e^{\lambda(t-T)}}{t^3} = g(t) \left( \lambda - \frac{3}{2t} \right), \label{eq:8_sub1} \\
g''(t) &= g'(t)\left(\lambda - \frac{3}{2t}\right) + g(t)\left(\frac{3}{2t^2}\right) = g(t) \left( \lambda^2 - \frac{3\lambda}{t} + \frac{15}{4t^2} \right). \label{eq:8_sub2}
\end{align}
Since $\lambda \ge 1$ (recall \eqref{eq:deif-lambda}) and $t \ge T \ge 2$, we have $\lambda - \frac{1}{t} \ge 1 - \frac{1}{2} = \frac{1}{2} > 0$, so
\begin{equation}\label{eq:9}
h'(t) = h(t) \left( \lambda - \frac{1}{t} \right) \ge \frac{1}{2} h(t) > 0. \quad 
\end{equation}

Define now
\begin{equation}\label{eq:10}
F := \lambda h' - h'', \quad H := (n-1)g + \lambda g' - g''. \quad 
\end{equation}
By \eqref{eq:2} and \eqref{eq:4} and recalling that $\lambda = n-2$, we have:
\begin{equation} \label{eq:def-laplace-u}
\begin{aligned}
-\Delta u^0  &\overset{\eqref{eq:6}}{=} -\Delta h \overset{\eqref{eq:2}}{=} -\frac{h'' - \lambda h'}{r^2} \overset{\eqref{eq:10}}{=} \frac{F}{r^2}, \\ 
-\Delta u^k &\overset{\eqref{eq:6}}{=} -\Delta(g\theta_k) \overset{\eqref{eq:4}}{=} -\frac{g'' - \lambda g' - (n-1)g}{r^2}\theta_k \overset{\eqref{eq:10}}{=} \frac{H}{r^2} \theta_k, \quad k = 1, \dots, n. \quad 
\end{aligned}
\end{equation}

Inserting \eqref{eq:7_sub1}--\eqref{eq:8_sub2} into the definition of $F$ and $H$ in \eqref{eq:10} yields:
\begin{align}
F &= \lambda h\left(\lambda - \frac{1}{t}\right) - h\left(\lambda^2 - \frac{2\lambda}{t} + \frac{2}{t^2}\right) = h \left( \frac{\lambda}{t} - \frac{2}{t^2} \right), \label{eq:12} \quad  \\
H &= (n-1)g + \lambda g\left(\lambda - \frac{3}{2t}\right) - g\left(\lambda^2 - \frac{3\lambda}{t} + \frac{15}{4t^2}\right) = g \left( n - 1 + \frac{3\lambda}{2t} - \frac{15}{4t^2} \right). \label{eq:13} \quad 
\end{align}
\vspace{0.5cm}

\textbf{Step 3: Definition of the anti-symmetric potential $\Omega$.} 

Set
\begin{equation}\label{eq:14}
 \quad D(t) := \frac{H(t)}{h'(t)}, \quad \quad C(t) := \frac{F(t) + g'(t)D(t)}{(n-1)g(t)}. \quad 
\end{equation}
Since (recall equation \eqref{eq:definition-h-g})
\begin{equation} \label{eq:g-durch-g} 
    g/h = t^{-1/2},
\end{equation}
the formulas \eqref{eq:7_sub1}--\eqref{eq:13} give explicitly:
\begin{equation}\label{eq:15}
D(t) \overset{\eqref{eq:7_sub1}+\eqref{eq:14}}{=} \frac{g(t)\left(n-1 + \frac{3\lambda}{2t} - \frac{15}{4t^2}\right)}{h(t)\left(\lambda - \frac{1}{t}\right)} \overset{\eqref{eq:definition-h-g}}{=} t^{-1/2} \frac{n-1 + \frac{3\lambda}{2t} - \frac{15}{4t^2}}{\lambda - \frac{1}{t}}, \quad 
\end{equation}
and since $F(t) = g(t) \left( \lambda t^{-1/2} - 2 t^{-3/2} \right)$ (see \eqref{eq:12}) and using \eqref{eq:8_sub1} we can rewrite $C$ defined in \eqref{eq:14} as:
\begin{equation}\label{eq:16}
C(t) = \frac{1}{n-1} \left[ \lambda t^{-1/2} - 2 t^{-3/2} + \left( \lambda - \frac{3}{2t} \right) D(t) \right]. \quad 
\end{equation}
Since $t \ge T \ge 2$, the denominator $\lambda - 1/t \ge 1/2$ in equation \eqref{eq:15} is bounded away from zero, and hence both functions $D$ and $C$ remain bounded by a constant depending only on $n$. Thus,
\begin{equation}\label{eq:17}
|C(t)| + |D(t)| \le c_n t^{-1/2} \quad \text{for every } t \ge T \ge 2. \quad 
\end{equation}

For $k = 1, \dots, n$, introduce the scalar valued 1-forms
\begin{equation}\label{eq:18}
\omega_k := C(t) d\theta_k + D(t) \theta_k \frac{dr}{r}. \quad 
\end{equation}
We define $\Omega = (\Omega_{ij})_{0 \le i, j \le n}$ by
\begin{equation}\label{eq:19}
\Omega_{0k} = \omega_k, \quad \Omega_{k0} = -\omega_k, \quad \Omega_{ij} = 0 \text{ otherwise}. \quad 
\end{equation}
Thus, $\Omega_{ij} = -\Omega_{ji}$ by definition and hence  $\Omega(x) \in \mathfrak{so}(n+1) \otimes \wedge^1$ for any $x \neq 0$.
\vspace{0.5cm}

\textbf{Step 4: Equation \eqref{eq: model-pde} holds pointwise a.e.} 

Due to equation \eqref{eq:3}, we have
\begin{equation}\label{eq:20}
du^0 = -h' \frac{dr}{r}, \quad du^k = g d\theta_k - g' \theta_k \frac{dr}{r}. \quad 
\end{equation}
We now use the following notation: $(\Omega \cdot \nabla u)^i = \sum_{k=0}^n \langle \Omega_{ik}, du^k \rangle$:

\textbf{Case $i=0$:} Consider the inner product with $du^k$; using orthogonality $\langle d\theta_k, dr \rangle = 0$ gives:
\begin{equation}
\begin{aligned} \label{eq:case-i-zero}
(\Omega \cdot \nabla u)^0 &= \sum_{k=1}^n \left\langle C d\theta_k + D \theta_k \frac{dr}{r}, g d\theta_k - g' \theta_k \frac{dr}{r} \right\rangle \quad  \\
&= C g \sum_{k=1}^n |d\theta_k|^2 - D g' \sum_{k=1}^n \theta_k^2 \left| \frac{dr}{r} \right|^2 + 0 \quad  \\
&= C g \left(\frac{n-1}{r^2}\right) - D g' (1) \left(\frac{1}{r^2}\right) = \frac{(n-1)Cg - Dg'}{r^2}. \quad 
\end{aligned}
\end{equation}
By definition of $C(t)$ in \eqref{eq:14},
\begin{equation}
    \eqref{eq:14} \quad \Leftrightarrow (n-1)Cg = F + g'D \quad \Leftrightarrow \quad  (n-1)Cg - Dg' = F
\end{equation}
holds equivalently. \\
Hence, we have
\begin{equation}
(\Omega \cdot \nabla u)^0 \overset{\eqref{eq:case-i-zero}}{=} \frac{F}{r^2} \overset{\eqref{eq:def-laplace-u}}{=} -\Delta u^0. \quad 
\end{equation}

\textbf{Case $i \in \{1, \dots, n\}$}: The only non-zero component is $\Omega_{i0} = -\omega_i$. Therefore:
\begin{equation}
\begin{aligned}
(\Omega \cdot \nabla u)^i & = \langle -\omega_i, du^0 \rangle \overset{\eqref{eq:18}}{=} \left\langle -C d\theta_i - D \theta_i \frac{dr}{r}, -h' \frac{dr}{r} \right\rangle \quad  \\
&= C h' \left\langle d\theta_i, \frac{dr}{r} \right\rangle + D h' \theta_i \left| \frac{dr}{r} \right|^2 = 0 + \frac{D h'}{r^2} \theta_i. \quad 
\end{aligned}
\end{equation}
By definition $D = H/h'$, so $D h' = H$. Thus:
\begin{equation}
(\Omega \cdot \nabla u)^i = \frac{H}{r^2} \theta_i \overset{\eqref{eq:def-laplace-u}}{=} -\Delta u^i. \quad 
\end{equation}
We have shown that
\begin{equation}\label{eq:21}
-\Delta u = \Omega \cdot \nabla u \quad \text{on } B^n \setminus \{0\}\quad 
\end{equation}
holds pointwise.

\vspace{0.5cm}

\textbf{Step 5: $\Omega \in L^n(\mathfrak{so}(n+1)\otimes \wedge^1)$}.

By the computations of Step 1 together with the explicit definition in \eqref{eq:18} we have:
\begin{equation}
\begin{aligned} \label{eq:only-omega}
\sum_{k=1}^n |\omega_k|^2 = C^2 \sum_{k=1}^n |d\theta_k|^2 + D^2 \sum_{k=1}^n \theta_k^2 \left|\frac{dr}{r}\right|^2 \overset{\eqref{eq:equal-one}+\eqref{eq:abs-value-theta-k}}{=} \frac{(n-1)C^2 + D^2}{r^2}. \quad 
\end{aligned}
\end{equation}
Since each $\omega_k$ appears twice in $\Omega$ (at the tuple $(0,k)$ and $(k,0)$):
\begin{equation}\label{eq:22}
|\Omega|^2 = \sum_{i,j=0}^n |\Omega_{ij}|^2 = 2 \sum_{k=1}^n |\omega_k|^2 \overset{\eqref{eq:only-omega}}{=}  \frac{2((n-1)C^2 + D^2)}{r^2} \le \frac{c_n}{r^2 t}, \quad 
\end{equation}
where we used estimate \eqref{eq:17}. Using spherical coordinates ($dx = r^{n-1} dr d\theta$), integrating  $|\Omega|^n$ over $B^n$ then yields
\begin{equation}\label{eq:23}
\begin{aligned}
\int_{B^n} |\Omega|^n \; dx &\le c_n |S^{n-1}| \int_0^1 \left(\frac{1}{r \sqrt{t}}\right)^n r^{n-1} dr = c_n |S^{n-1}| \int_0^1 \frac{dr}{r t^{n/2}}. \quad 
\end{aligned}
\end{equation}
Substituting $t = T - \log r $, $ dt = -\frac{dr}{r}$ and $r=1$ gives $ t=T$, so  $r \to 0 $ gives $ t \to \infty$ and hence
\begin{equation}
\int_0^1 \frac{dr}{r t^{n/2}} = \int_T^\infty t^{-n/2} dt = \left[ \frac{t^{1 - n/2}}{1 - n/2} \right]_T^\infty = \frac{T^{1 - n/2}}{\frac{n}{2} - 1} < \infty \quad \text{for } n \ge 3. \quad 
\end{equation}
Thus, $\|\Omega\|_{L^n(B^n)} \le c_n T^{\frac{1}{n} - \frac{1}{2}} \to 0$ as $T \to \infty$. So for $T$ large enough, we obtain  $\|\Omega\|_{L^n(B^n)} < \varepsilon_0$.
\vspace{0.5cm}

\textbf{Step 6. $ u\in W^{1\frac{n}{n-1}}(B^n, \mathbb{R}^{n+1})$.} 

We have
\begin{equation}\label{eq:24}
|\nabla u|^2 = |\nabla u^0|^2 + \sum_{k=1}^n |\nabla u^k|^2 \overset{\eqref{eq:20}}{=} \frac{(h')^2 + (g')^2 + (n-1)g^2}{r^2}. \quad 
\end{equation}
Since $g = h t^{-1/2}$ (recall \eqref{eq:definition-h-g}), due to equations \eqref{eq:7_sub1} and \eqref{eq:8_sub1} we have
\begin{equation}
    (h')^2 + (g')^2 + (n-1)g^2  \approx \lambda^2 h^2,
\end{equation} and hence 
\begin{equation}\label{eq:25}
|\nabla u(x)| \sim \frac{h(t)}{r} \overset{\eqref{eq:5}}{=} \frac{r^{2-n}}{r t} = \frac{1}{r^{n-1} t}. \quad 
\end{equation}
Using $(n-1)\frac{n}{n-1} = n$, we get
\begin{equation}
\begin{aligned}
\int_{B^n} |\nabla u|^{\frac{n}{n-1}} dx &\le c_n |S^{n-1}| \int_0^1 r^{n-1} \left( \frac{1}{r^{n-1} t} \right)^{\frac{n}{n-1}} dr \\
&= c_n |S^{n-1}| \int_0^1 \frac{dr}{r t^{\frac{n}{n-1}}} = c_n |S^{n-1}| \int_T^\infty t^{-\frac{n}{n-1}} dt < \infty, \quad 
\end{aligned}
\end{equation}
since $ \frac{n}{n-1} > 1$ for all $n \ge 3$.

The computation remains for the $L^{\frac{n}{n-1}}$–norm of $u$. Due to the definition of $u$ in \eqref{eq:6}, we have
\begin{equation}
    \begin{aligned}
        |u|^2 = h^2 + g^2 = h^2(1 + t^{-1}) \le 2h^2 \text{  for } t \ge 2.
    \end{aligned}
\end{equation} Thus:
\begin{equation}
\int_{B^n} |u|^{\frac{n}{n-1}} dx \overset{\eqref{eq:definition-h-g}}{\le} c_n \int_0^1 r^{n-1} (r^{2-n} t^{-1})^{\frac{n}{n-1}} dr = c_n \int_0^1 r^{n-1 - (n-2)p} t^{-\frac{n}{n-1}} \; dr. \quad 
\end{equation}
For the radial exponent we have $n - 1 - (n-2)\frac{n}{n-1} = \frac{(n-1)^2 - n(n-2)}{n-1} = \frac{1}{n-1} > -1$, so the integrand is integrable near $r=0$. Hence, $u \in W^{1\frac{n}{n-1}}(B^n, \mathbb{R}^{n+1})$.

\vspace{0.5cm}

\textbf{Step 7: Removability of the singularity.} 

To verify that no Dirac mass is created at the origin, let $\varphi \in C_c^\infty(B^n, \mathbb{R}^{n+1})$. Integrating by parts on $B^n \setminus \overline{B_\rho}$ yields 
\begin{equation}
\int_{B^n \setminus \overline{B_\rho}} \langle \nabla u, \nabla \varphi \rangle dx = \int_{B^n \setminus \overline{B_\rho}} (\Omega \cdot \nabla u) \cdot \varphi \, dx + \int_{\partial B_\rho} \partial_r u \cdot \varphi \, d\mathcal{H}^{n-1}. \quad 
\end{equation}
By \eqref{eq:20}, we have $|\partial_r u| \le \frac{|h'| + |g'|}{r}$. The integral over the boundary leads to
\begin{equation}
\int_{\partial B_\rho} |\partial_r u| d\mathcal{H}^{n-1} \le c_n \rho^{n-1} \frac{|h'(t_\rho)| + |g'(t_\rho)|}{\rho} = c_n \rho^{n-2} \left( |h'(t_\rho)| + |g'(t_\rho)| \right),
\end{equation}
where $t_\rho = T + \log(1/\rho)$.

Recalling that $e^{\lambda(t_\rho - T)} = e^{\lambda \log(1/\rho)} = \rho^{-\lambda} = \rho^{-(n-2)}$, we substitute $h'(t_\rho)$ and $g'(t_\rho)$ from \eqref{eq:7_sub1} and \eqref{eq:8_sub1} in order to get:
\begin{equation}
h'(t_\rho) = \rho^{-(n-2)} t_\rho^{-1} \left(\lambda - \frac{1}{t_\rho}\right), \quad g'(t_\rho) = \rho^{-(n-2)} t_\rho^{-3/2} \left(\lambda - \frac{3}{2t_\rho}\right).
\end{equation}
Multiplying by $\rho^{n-2}$ then gives
\begin{equation}
\rho^{n-2} h'(t_\rho) = t_\rho^{-1} \left(\lambda - \frac{1}{t_\rho}\right), \quad \rho^{n-2} g'(t_\rho) = t_\rho^{-3/2} \left(\lambda - \frac{3}{2t_\rho}\right).
\end{equation}
As $\rho \to 0^+$, $t_\rho = T + \log(1/\rho) \to +\infty$ and thus:
\begin{equation}
\lim_{\rho \to 0} \int_{\partial B_\rho} |\partial_r u| d\mathcal{H}^{n-1} \le c_n \lim_{t_\rho \to \infty} \left( \frac{\lambda - \frac{1}{t_\rho}}{t_\rho} + \frac{\lambda - \frac{3}{2t_\rho}}{t_\rho^{3/2}} \right) = 0. \quad 
\end{equation}
This shows that $\rho \to 0$ yields 
\begin{equation}
    \int_{B^n} \langle \nabla u, \nabla \varphi \rangle dx = \int_{B^n} (\Omega \cdot \nabla u) \cdot \varphi \, dx,
    \end{equation}
    and we conclude that $-\Delta u = \Omega \cdot \nabla u$ holds in  $\mathcal{D}'(B^n)$.

\vspace{0.5cm}

\textbf{Step 8: $u \notin L^\infty$ and is therefore not continuous.} 

Recall  the definition of the first component of $u$ (see \eqref{eq:6}):
\begin{equation}
u^0(x) = h(t) = \frac{|x|^{2-n}}{T + \log(1/|x|)}. \quad 
\end{equation}
Since $n \ge 3$, as $x \to 0$, $|x|^{2-n}$ grows polynomially, $T + \log(1/|x|)$ grows logarithmically. Hence,
\begin{equation}
u^0(x) \rightarrow +\infty \quad \text{as } x \to 0. \quad 
\end{equation}
Thus, $u \notin L^\infty$ and cannot have a continuous representative.
\end{proof}

\begin{remark}
The  construction above has the exact asymptotic behaviour:
\begin{equation}
|\Omega(x)| \sim \frac{1}{|x| (T + \log(1/|x|))^{1/2}}, \quad |\nabla u(x)| \sim \frac{1}{|x|^{n-1} (T + \log(1/|x|))}. \quad 
\end{equation}
By standard rearrangement characterisation of Lorentz spaces, we have
\begin{equation}
\Omega \in \left( \bigcap_{q > 2} L^{(n,q)}(B^n) \right) \setminus L^{(n,2)}(B^n), \quad 
\end{equation}
whereas
\begin{equation}
\nabla u \in \left( \bigcap_{s > 1} L^{(\frac{n}{n-1},s)}(B^n) \right) \setminus L^{(\frac{n}{n-1},1)}(B^n). \quad 
\end{equation}
Thus, the result is optimal.
\end{remark}

\appendix

\setcounter{equation}{0}
\setcounter{figure}{0}
\setcounter{table}{0}

\renewcommand{\theequation}{A.\arabic{equation}}
\renewcommand{\thefigure}{A.\arabic{figure}}
\renewcommand{\thetable}{A.\arabic{table}}

\section{}

Let $1 \leq q \le 2$ and denote by $2 \leq q'$ its corresponding conjugate exponent. 
In the following, let $\frac{n}{2}<p<n$. Then we have $W^{1,(p,1)} \hookrightarrow L^{(s,1)} \subset L^{(\frac{np}{n-p},q')}$ for any $s \leq \frac{np}{n-p}$ and hence $n< \frac{np}{n-p}$ is satisfied by the restriction on the parameter $p$. Further, since $p<n$ we have 
\begin{equation}
    \begin{aligned}
        \frac{1}{p}= \frac{1}{n} + \frac{1}{\frac{np}{n-p}}, \quad \text{ and } \quad 1= \frac{1}{q}+ \frac{1}{q'},\label{eq:exponent-p-np}
    \end{aligned}
\end{equation}
which for any $ f\in L^{(n,q)}$ and $g \in W^{1,(p,1)}$ leads to
\begin{equation}
    \begin{aligned}
        \Vert f\cdot g \Vert_{L^{(p,1)}} \lesssim \Vert f \Vert_{L^{(n,q)}} \Vert g\Vert_{L^{(\frac{np}{n-p},q')}} \lesssim \Vert f \Vert_{L^{(n,q)}} \Vert g\Vert_{W^{1,(p,1)}}. \label{eq:hölder-inequality-for-p}
    \end{aligned}
\end{equation}
\begin{lemma}\label{lm:better-regularity-space}
  Assume there exists $\varepsilon>0$ and a constant $C>0$ such that for any $\Omega \in W^{1,(p,1)}(B^n, \mathfrak{so}(m) \otimes \wedge^1)$ satisfying 
    \begin{equation} \begin{aligned}
        \Vert \Omega \Vert_{L^{(n,q)}} < \varepsilon(n,m),
    \end{aligned} \end{equation} 
    there exists $P \in W^{2,(p,1)}(B^n,\mathrm{SO}(m))$ and $\xi \in W^{2,(p,1)}(B^n, \mathfrak{so}(m) \otimes \bigwedge^{n-2})$ such that 
    \begin{equation} \begin{aligned}
      \begin{cases}
            \ast d\xi = P^{-1} dP + P^{-1} \Omega P & \text{ in } B^n, \\
            d^\ast\xi = 0 & \text{ in } B^n, \\
            \xi = 0 & \text{ on } \partial B^n, 
        \end{cases} \label{eq:gauge-system-better}
    \end{aligned} \end{equation} 
    and satisfy
    \begin{equation} \begin{aligned}
       \Vert dP \Vert_{W^{1,(p,1)}} + \Vert d\xi \Vert_{W^{1,(p,1)}} \leq C \Vert \Omega \Vert_{W^{1,(p,1)}}
    \label{eq:apriori-bound-better-lemma1}
    \end{aligned} \end{equation} 
    and
    \begin{equation} \begin{aligned}
       \Vert dP \Vert_{L^{(n,q)}} + \Vert d \xi \Vert_{L^{(n,q)}}\leq C \Vert \Omega \Vert_{L^{(n,q)}} \leq C \varepsilon(n,m).
    \label{eq:apriori-bound-better-lemma2}
    \end{aligned} \end{equation} 
\end{lemma}

\begin{proof}[Proof of Lemma \ref{lm:better-regularity-space}.] \label{proof:better-space}
We follow the continuity method by Uhlenbeck \cite{uhlenbeck} with the techniques established by Rivière \cite{conservation-law-conf-invariant}. 

Introduce for any constants $\varepsilon>0$ and $C\geq0$ the set 
\begin{equation}
\begin{aligned}
   \mathcal{U}_{\varepsilon,C} := \left\{ 
\begin{array}{l}
\Omega \in W^{1,(p,1)}(B^n, \mathfrak{so}(m) \otimes \wedge^1(\R^n)):\; \|\Omega\|_{L^{(n,q)}} \le \varepsilon, \text{ and} \\[1ex]
\text{there exist } P \text{ and } \xi \text{ satisfying } \eqref{eq:gauge-system-better},\eqref{eq:apriori-bound-better-lemma1}, \eqref{eq:apriori-bound-better-lemma2}.
\end{array} 
\right\}
\end{aligned} \label{eq:set-u-wishlist}
\end{equation}
Further, define the set 
\begin{equation}
    \begin{aligned}
        \mathcal{V}_\varepsilon := \left\{ \Omega \in W^{1,(p,1)}(B^n, \mathfrak{so}(m) \otimes \wedge^1 ): \;\|\Omega\|_{L^{(n,q)}} \le \varepsilon \right\}. \label{eq:set-v}
    \end{aligned}
\end{equation}

We show that for $\varepsilon>0$ small enough and $C\geq0$ large enough, $\mathcal{V}_\varepsilon= \mathcal{U}_{\varepsilon,C}$ by the following steps: 
\begin{itemize}[label=]
\item 1. Step: \quad $\mathcal{U}_{\varepsilon,C}$ is non-empty ($0\in \mathcal{U}_{\varepsilon,C}$).
    \item 2. Step: \quad $\mathcal{V}_\varepsilon$ is path-connected. 
    \item 3. Step:  \quad$\mathcal{U}_{\varepsilon,C}$ is closed in $\mathcal{V}_\varepsilon$ with respect to the $W^{1,(p,1)}$-topology.
    \item 4. Step: \quad $\mathcal{U}_{\varepsilon,C}$ is open in $\mathcal{V}_\varepsilon$ with respect to the $W^{1,(p,1)}$-topology.
\end{itemize}
\vspace{0.5cm}

\textbf{1. Step: $\mathcal{U}_{\varepsilon,C}$ is non-empty.} 

Trivially $0\in \mathcal{U}_{\varepsilon,C}$ with the corresponding solution 
\[P=\mathrm{Id}_m \in W^{2,(p,1)}, \qquad \xi =0 \in W^{2,(p,1)}.
\]

\vspace{0.5cm}

\textbf{2. Step: $\mathcal{V}_{\varepsilon}$ is path-connected.}

We show that $\mathcal{V}_\varepsilon$ is star-shaped and hence path-connected. Fix $\Omega_0 \in \mathcal{V}_\varepsilon$ and let $\Omega \in \mathcal{V}_\varepsilon$ be arbitrary. Define the path by 
\[
\Omega_t \coloneqq t \Omega + (1-t) \Omega_0, \qquad t \in [0,1].
\]
Then, for any $t \in [0,1]$, triangle inequality gives
\[
\Vert \Omega_t \Vert_{L^{(n,q)}} \leq t \Vert \Omega \Vert_{L^{(n,q)}} + (1-t) \Vert \Omega_0 \Vert_{L^{(n,q)}} \leq t \varepsilon + (1-t)\varepsilon = \varepsilon,
\]
and hence $\Omega_t \in \mathcal{V}_\varepsilon.$

\vspace*{0.5cm}

\textbf{3. Step: $\mathcal{U}_{\varepsilon,C}$ is closed with respect to the $W^{1,(p,1)}$-topology.}

Let $(\Omega_k)_{k \in \N}$ be a sequence in $\mathcal{U}_{\varepsilon,C}$ with 
\[
\Omega_k \to \Omega_\infty \quad \text{strongly in } W^{1,(p,1)},
\]
where the limit satisfies $\Omega_\infty \in \mathcal{V}_\varepsilon.$ We claim that $\Omega_\infty \in \mathcal{U}_{\varepsilon,C}.$

Since $\Omega_k \in \mathcal{U}_{\varepsilon,C}$, there exists a sequence of solutions to \eqref{eq:gauge-system-better}, denoted by $P_k \in W^{2,(p,1)}(B^n, \mathrm{SO}(m))$ and $\xi_k \in W^{2,(p,1)}(B^n, \mathfrak{so}(m) \otimes \bigwedge^{n-2})$ for each $k \in \N$. 
Due to the uniform bounds $\Vert \Omega_k \Vert_{L^{(n,q)}} \leq \varepsilon$ and $\Vert \Omega_k \Vert_{W^{1,(p,1)}} \lesssim \Vert \Omega_\infty \Vert_{W^{1,(p,1)}} < \infty$, together with estimates \eqref{eq:apriori-bound-better-lemma1} and \eqref{eq:apriori-bound-better-lemma2}, yields a weakly convergent subsequence $(P_{k_j}, \xi_{k_j})$ such that
\begin{equation} \begin{aligned}
    P_{k_j} \rightharpoonup P_\infty \quad \text{and} \quad \xi_{k_j} \rightharpoonup \xi_\infty \quad \text{weakly in } W^{2,(p,1)}.
\end{aligned} \end{equation}
By lower semicontinuity of the Sobolev-Lorentz norms, the weak limits satisfy
\begin{equation}
    \begin{aligned}
        \Vert dP_\infty \Vert_{W^{1,(p,1)}} + \Vert d\xi_\infty \Vert_{W^{1,(p,1)}} \leq \liminf_{j \to \infty} C \Vert \Omega_{k_j} \Vert_{W^{1,(p,1)}} = C \Vert \Omega_\infty \Vert_{W^{1,(p,1)}},
    \end{aligned}
\end{equation}
and 
\begin{equation}
    \begin{aligned}
        \Vert dP_\infty \Vert_{L^{(n,q)}} + \Vert d\xi_\infty \Vert_{L^{(n,q)}} \leq \liminf_{j \to \infty} C \Vert \Omega_{k_j} \Vert_{L^{(n,q)}} = C \Vert \Omega_\infty \Vert_{L^{(n,q)}}.
    \end{aligned}
\end{equation}

It remains to show that $(P_\infty, \xi_\infty)$ solves \eqref{eq:gauge-system-better} with respect to $\Omega_\infty$ and that $P_\infty \in \mathrm{SO}(m)$ almost everywhere. \\
First, by Rellich-Kondrachov compact embedding, $W^{2,(p,1)}(B^n) \hookrightarrow W^{1,(s,1)}(B^n)$ for any $1 \le s < \frac{np}{n-p}$, there exists a strongly convergent subsequence in $L^{(s,1)}$, which implies pointwise convergence almost everywhere of $P_{k_j} \to P_\infty$ and $d^\ast\xi_{k_j} \to d^\ast\xi_\infty$. Passing to the limit in $P_{k_j} P_{k_j}^\top = \mathrm{Id}_m$ and $d^\ast\xi_{k_j} = 0$ pointwise a.e. confirms that $P_\infty \in W^{2,(p,1)}(B^n, \mathrm{SO}(m))$ and $d^\ast\xi_\infty = 0$ on $B^n$ almost everywhere. \\

Second, we show that the limit is a solution in the sense of distributions, i.e.,
\begin{equation}
\begin{aligned}
 \ast d \xi_\infty = P_\infty^{-1} dP_\infty + P_\infty^{-1} \Omega_\infty P_\infty \quad \text{in } \mathcal{D}'(B^n). \label{eq:weak-sense-closed-proof}
 \end{aligned} 
\end{equation}
The compact embedding $W^{2,(p,1)}(B^n) \hookrightarrow W^{1,s}(B^n)$ ensures strong convergence:
\[
 P_{k_j} \to P_\infty \quad \text{in } L^{(s,1)}(B^n) \quad \text{and} \quad dP_{k_j} \to dP_\infty \quad \text{strongly in } L^{(s,1)}(B^n) \quad \forall 1 \leq s < \frac{np}{n-p}.
\]

\vspace{0.5cm}
\textit{Left-Hand Side (LHS) of \eqref{eq:weak-sense-closed-proof}:}
For any test function $\Phi \in C_c^\infty(B^n, \mathfrak{so}(m) \otimes \wedge^1(\R^n))$, weak convergence $\xi_{k_j} \rightharpoonup \xi_\infty$ in $W^{2,(p,1)}$ implies $d\xi_{k_j} \rightharpoonup d\xi_\infty$ weakly in $W^{1,(p,1)}$ and thus in $L^{(n,q)}$:
\begin{equation} \begin{aligned}
    \left| \int_{B^n} (\ast d \xi_{k_j} - \ast d \xi_\infty) \cdot \Phi \right| \longrightarrow 0.
\end{aligned} \end{equation}

\vspace{0.5cm}

\textit{Right-Hand Side (RHS) of \eqref{eq:weak-sense-closed-proof}:}
Testing against $\Phi \in C_c^\infty(B^n, \mathfrak{so}(m) \otimes \wedge^1(\R^n))$, we bound each term:
\small
\begin{equation} \begin{aligned}
    & \Bigg\vert \int_{B^n} (P_{\infty}^{-1} dP_\infty - P_{k_j}^{-1} dP_{k_j}) \Phi \Bigg\vert \\
    & \le \Bigg\vert \int_{B^n} P_{\infty}^{-1} (dP_\infty - dP_{k_j}) \Phi \Bigg\vert + \Bigg\vert \int_{B^n} (P_{k_j}^{-1} - P_\infty^{-1}) dP_{k_j} \Phi \Bigg\vert \\
    & \lesssim \Bigg\vert \int_{B^n} (dP_\infty - dP_{k_j}) (P_\infty^{-\top} \Phi) \Bigg\vert + \Vert dP_{k_j} \Vert_{L^{(n,q)}} \Vert P_{k_j}^{-1} - P_\infty^{-1} \Vert_{L^{(\frac{n}{n-1},q')}} \Vert \Phi \Vert_{L^\infty} \\
    & \longrightarrow 0,
\end{aligned} \end{equation}
\normalsize
and similarly
\small
\begin{equation} \begin{aligned}
    & \Bigg\vert \int_{B^n} (P_{\infty}^{-1} \Omega_\infty P_\infty - P_{k_j}^{-1} \Omega_{k_j} P_{k_j}) \Phi \Bigg\vert \\
    & \le \Bigg\vert \int_{B^n} P_{\infty}^{-1} \Omega_\infty (P_\infty - P_{k_j}) \Phi \Bigg\vert + \Bigg\vert \int_{B^n} (P_{k_j}^{-1}\Omega_{k_j} - P_\infty^{-1} \Omega_\infty) P_{k_j} \Phi \Bigg\vert \\
    & \lesssim \Vert P_\infty^{-1} \Vert_{L^\infty} \Vert \Omega_\infty \Vert_{L^{(n,q)}} \Vert P_\infty - P_{k_j} \Vert_{L^{(\frac{n}{n-1},q')}} \Vert \Phi \Vert_{L^\infty} \\
    & \quad + \Vert P_{k_j} \Vert_{L^\infty} \Big( \Vert P_{k_j}^{-1} \Vert_{L^\infty} \Vert (\Omega_{k_j} - \Omega_\infty) \Phi \Vert_{L^1} + \Vert P_{k_j}^{-1} - P_\infty^{-1} \Vert_{L^{(\frac{n}{n-1},q')}} \Vert \Omega_\infty \Vert_{L^{(n,q)}} \Vert \Phi \Vert_{L^\infty} \Big) \\
    & \longrightarrow 0. 
\end{aligned}\end{equation}
\normalsize

Using the trace theorem for Sobolev--Lorentz spaces, the continuous embedding 
\begin{equation}
    W^{2,(p,1)}(B^n) \hookrightarrow W^{2-\frac{1}{p},p}(\partial B^n) \hookrightarrow L^1(\partial B^n)
\end{equation} 
ensures that $\xi_k \big|_{\partial B^n} = 0$ passes to the weak limit, so $\xi_\infty \big|_{\partial B^n} = 0.$

Hence, $\Omega_\infty$ fulfils each condition for belonging to $\mathcal{U}_{\varepsilon,C}$.

\vspace*{0.5cm}

\textbf{4. Step: \quad $\mathcal{U}_{\varepsilon,C}$ is open in $\mathcal{V}_\varepsilon$ with respect to the $W^{1,(p,1)}$-topology.}

In this step, we need to restrict our choice of $\varepsilon>0$, choosing it sufficiently small, and set $C\geq 0$ to be sufficiently large.

Let $\Omega_0 \in \mathcal{U}_{\varepsilon,C}.$ We claim that there exists $\delta>0$ such that for any $\omega \in W^{1,(p,1)}(B^n, \mathfrak{so}(m) \otimes \wedge^1)$
\begin{equation}
    \begin{aligned}     
\text{with } \Vert \omega \Vert_{W^{1,(p,1)}} < \delta, \quad  \text{ we have } \quad \Omega_0 + \omega \in \mathcal{U}_{\varepsilon,C}.
    \end{aligned}
\end{equation}

Since $\Omega_0 \in \mathcal{U}_{\varepsilon,C}$, there exists a corresponding pair of solutions $(P_0, \xi_0)$ satisfying \eqref{eq:gauge-system-better}. 
Define the non-linear operator 
\small
\begin{equation}
    \begin{aligned}
        \mathcal{N} \colon W^{1,(p,1)}(B^n, \mathfrak{so}(m) \otimes \wedge^1) \times X &\to L^{(p,1)}(B^n,\mathfrak{so}(m)) \times W^{1-\frac{1}{p},(p,1)}(\partial B^n, \mathfrak{so}(m)), \\
        (\lambda, U) &\mapsto \Bigg( d^\ast \Big( e^{-U}de^U + e^{-U}(\ast d\xi_0 + \lambda ) e^U \Big), \; \partial_r U \Big|_{\partial B^n} \Bigg),
    \end{aligned}
\end{equation}
where 
\begin{equation}
\begin{aligned}
X &:= \left\{ U \in W^{2,(p,1)}(B^n, \mathfrak{so}(m)) \;\middle|\; \int_{B^n} U \, dx = 0 \right\} \\
     Y &:= L^{(p,1)}(B^n, \mathfrak{so}(m)) \times W^{1-\frac{1}{p},(p,1)}(\partial B^n, \mathfrak{so}(m)).
     \end{aligned}
\end{equation}
\normalsize
 
Since $P \in W^{2,(p,1)}(B^n, \mathrm{SO}(m))$, the map $\lambda \mapsto P^{-1}\lambda P$ and its inverse $\zeta \mapsto P \zeta P^{-1}$ are continuous from $W^{1,(p,1)}$ to $W^{1,(p,1)}$. Indeed, because $P(x) \in \mathrm{SO}(m)$ pointwise, we have
\begin{equation}
    \begin{aligned}
        \Vert P \lambda P^{-1} \Vert_{L^{(p,1)}} \lesssim \Vert \lambda \Vert_{L^{p}},
    \end{aligned}
\end{equation}
and applying the product rule yields:
\begin{equation}
    \begin{aligned}
        \Vert \nabla ( P \lambda P^{-1} )\Vert_{L^{(p,1)}} 
        &\lesssim \Vert \nabla P \lambda P^{-1} \Vert_{L^{(p,1)}} + \Vert P \nabla \lambda P^{-1} \Vert_{L^{(p,1)}} + \Vert P \lambda \nabla P^{-1} \Vert_{L^{(p,1)}} \\
        &\lesssim \Vert \nabla P \Vert_{L^{(n,q)}} \Vert \lambda \Vert_{L^{(\frac{np}{n-p},q')}} + \Vert \nabla \lambda \Vert_{L^{(p,1)}} + \Vert \nabla P \Vert_{L^{(n,q)}} \Vert \lambda \Vert_{L^{(\frac{np}{n-p},q')}} \\
        &\lesssim \Vert \nabla P \Vert_{L^{(n,q)}} \Vert \lambda \Vert_{W^{1,(p,1)}} + \Vert \nabla \lambda \Vert_{L^{(p,1)}},
    \end{aligned}
\end{equation}
where we used $W^{1,(p,1)}(B^n) \hookrightarrow L^{(\frac{np}{n-p},q')}(B^n)$.

Thus, the non-linear map $\mathcal{N}$ is $C^1$-smooth. We follow the strategy of Lemma 2.7 and Lemma 2.8 in \cite{uhlenbeck}.

To derive the linearized operator $L_1 := \partial_U \mathcal{N}(0,0) \colon X \to Y$, we compute the directional derivative of $\mathcal{N}(0, tV)$ at $t=0$ for $V \in X$. Expanding the matrix exponentials yields
\begin{equation}
\begin{aligned}
    e^{tV} &= \mathrm{Id}_m + tV + \mathcal{O}(t^2), \\
    e^{-tV} &= \mathrm{Id}_m - tV + \mathcal{O}(t^2), \\
    d(e^{tV}) &= t\, dV + \mathcal{O}(t^2).
\end{aligned}
\end{equation}

Substituting these expansions into the first component of $\mathcal{N}$ gives
\begin{equation}
\begin{aligned}
    e^{-tV}d(e^{tV}) + e^{-tV} \ast d\xi_0 e^{tV} 
    &= (\mathrm{Id}_m - tV)(t\, dV) + (\mathrm{Id}_m - tV) \ast d\xi_0 (\mathrm{Id}_m + tV) + \mathcal{O}(t^2) \\
    &= \ast d\xi_0 + t \Big( dV + \ast d\xi_0 V - V \ast d\xi_0 \Big) + \mathcal{O}(t^2).
\end{aligned}
\end{equation}

Differentiating with respect to $t$ at $t=0$ yields
\begin{equation}
    \left.\frac{d}{dt}\right|_{t=0} \Big( e^{-tV}d(e^{tV}) + e^{-tV} \ast d\xi_0 e^{tV} \Big) = dV + [\ast d\xi_0, V].
\end{equation}

Applying $d^\ast$ (noting that $d^\ast d V = \Delta V$), the linearized operator $L_1 := \partial_U \mathcal{N}(0,0) \colon X \to Y$ is explicitly given by
\begin{equation}
    L_1 V = \Big( \Delta V + d^\ast( \ast d\xi_0 V - V \ast d\xi_0 ), \; \partial_r V \Big|_{\partial B^n} \Big).
\end{equation}

Standard elliptic theory for the Laplacian $L_0 V := (\Delta V, \partial_r V \big|_{\partial B^n})$ in Sobolev--Lorentz spaces yields
\begin{equation}
    \Vert V \Vert_{W^{2,(p,1)}(B^n)} \le c_0 \, \Vert L_0 V \Vert_Y = c_0 \Big( \Vert \Delta V \Vert_{L^{(p,1)}(B^n)} + \Vert \partial_r V \Vert_{W^{1-\frac{1}{p},(p,1)}(\partial B^n)} \Big). \label{eq:elliptic-estimate-appendix}
\end{equation}

Combining $L_0$ and $L_1$ via the triangle inequality gives
\begin{equation}
    \Vert V \Vert_{W^{2,(p,1)}(B^n)} \le c_0 \Vert L_1 V \Vert_Y + c_0 \Vert   d\xi_0 \wedge dV -  dV \wedge  d\xi_0  \Vert_{L^{(p,1)}(B^n)}.
\end{equation}

It remains to estimate the perturbation term. Using  $W^{1,(p,1)}(B^n) \hookrightarrow L^{(\frac{np}{n-p},q')}(B^n)$ (see \eqref{eq:hölder-inequality-for-p}),  yields
    \begin{equation}
        \Vert  d\xi_0 \wedge d V \Vert_{L^{(p,1)}} \leq \Vert d\xi_0 \Vert_{L^{(n,q)}} \Vert \nabla V \Vert_{L^{(\frac{np}{n-p},q')}} \leq C_2 \Vert d\xi_0 \Vert_{L^{(n,q)}} \Vert V \Vert_{W^{2,(p,1)}}. \label{eq:wedge-estimate-appendix}
    \end{equation}

Inserting \eqref{eq:wedge-estimate-appendix} in \eqref{eq:elliptic-estimate-appendix} and using the anti-symmetry of the wedge-product, we obtain:
\begin{equation}
    \Vert V \Vert_{W^{2,(p,1)}(B^n)} \leq c_0\Vert L_1 V \Vert_Y + c_0 C_3 \Vert d\xi_0 \Vert_{L^{(n,q)}(B^n)} \Vert V \Vert_{W^{2,(p,1)}(B^n)}.
\end{equation}

Since $\xi_0$ is a solution of $\Omega_0 \in \mathcal
U_{\varepsilon,C}$, we have $\Vert d\xi_0 \Vert_{L^{(n,q)}(B^n)} \leq \varepsilon$ and choosing $\varepsilon < \frac{1}{2 c_0 C_3}$ small enough, we absorb the perturbation term into the left-hand side and get:
\begin{equation}\label{eq:apriori_bound}
    \Vert V \Vert_{W^{2,(p,1)}(B^n)} \leq 2 c_0 \Vert L_1 V \Vert_Y.
\end{equation}

If $L_1 V = 0$, estimate \eqref{eq:apriori_bound} immediately implies $\Vert V \Vert_{W^{2,(p,1)}} = 0$, so $V = 0$. Thus, $\ker(L_1) = \{0\}$.

It remains to show that $L_1$ is surjective and hence an isomorphism. Recall that $L_0 \colon X \to Y$ is a classical elliptic operator defining an isomorphism on $X$ due to the mean-zero assumption. The family $(L_t)_{t \in [0,1]}$, where
\begin{equation}
 t \mapsto  L_t := L_0 + t \,  K,  \qquad \text{ and } \;  K(V) = \big( \ast (d\xi_0 \wedge dV)- \ast (dV \wedge d\xi_0) , \; 0 \big)
\end{equation}
defines a continuous path of bounded operators. For $\varepsilon$ small enough, $L_t$ remains Fredholm for all $t \in [0,1]$. By the invariance of the Fredholm index under continuous perturbations, we obtain
\begin{equation}
    \mathrm{ind}(L_1) = \mathrm{ind}(L_0) = 0.
\end{equation}
By the definition of the index and since $\ker(L_1) = \{0\}$, it follows that $L_1$ is surjective.
\vspace{0.3cm}

Applying the Implicit Function Theorem gives a $\delta > 0$ and an open neighborhood $\mathcal{O}$ of $0$ in $X$ such that for every $\lambda \in W^{1,(p,1)}(B^n, \mathfrak{so}(m) \otimes \wedge^1)$ satisfying $\|\lambda\|_{W^{1,(p,1)}(B^n)} < \delta$, there exists a unique $U_\lambda \in \mathcal{O}$ with $\mathcal{N}(\lambda, U_\lambda) = 0$ and $\int_{B^n} U_\lambda \, dx = 0$.

\vspace{0.2cm}

Since the map $\lambda \mapsto U_\lambda$ is $C^1$-continuous with $U_0 = 0$, we have $\|U_\lambda\|_{W^{2,(p,1)}} \to 0$ as $\|\lambda\|_{W^{1,(p,1)}} \to 0$. Because $W^{2,(p,1)}(B^n)$ is a Banach algebra under matrix multiplication, setting $Q_\lambda := e^{U_\lambda}$ yields
\begin{align}
    \|d(Q_\lambda - \mathrm{Id}_m)\|_{W^{1,(p,1)}} \le C \|dU_\lambda\|_{W^{1,(p,1)}} < \alpha, \label{eq:small-identity-alpha}
\end{align}
for any given $\alpha > 0$, assuming that $\delta$ is chosen sufficiently small.

\vspace{0.3cm}
Since $(P_0, \xi_0)$ solves \eqref{eq:gauge-system-better} for $\Omega_0$, we define  $P_\lambda := P_0 Q_\lambda$ and obtain:
\begin{equation}
    \begin{aligned}
        \ast d\xi_\lambda &= Q_\lambda^{-1}dQ_\lambda + Q_{\lambda}^{-1} \big(P_0^{-1} dP_0 + P_0^{-1}\Omega_0 P_0 + \lambda \big) Q_{\lambda } \\
        &= (P_0 Q_\lambda)^{-1} d(P_0 Q_\lambda) + (P_0 Q_\lambda)^{-1} (\Omega_0 + P_0 \lambda P_0^{-1} ) (P_0 Q_\lambda) \\
        &= P_\lambda^{-1} dP_\lambda + P_\lambda^{-1} (\Omega_0 + \omega) P_\lambda,
    \end{aligned}
\end{equation}
where $\omega \coloneqq P_0 \lambda P_0^{-1}$. Thus, $(P_\lambda, \xi_\lambda)$ solves \eqref{eq:gauge-system-better} for $\Omega := \Omega_0 + \omega$, with
\begin{equation}
    \|\omega\|_{W^{1,(p,1)}} \le C \|\lambda\|_{W^{1,(p,1)}} < \delta. \label{eq:omega-delta}
\end{equation}

Moreover, standard elliptic estimates applied to
\[
    \ast d(\xi_\lambda - \xi_0) = Q_\lambda^{-1} dQ_\lambda + Q_\lambda^{-1}\big[ \ast d\xi_0, Q_\lambda - \mathrm{Id}_m \big] + Q_\lambda^{-1} \lambda Q_\lambda
\]
along with $d^\ast(\xi_\lambda - \xi_0) = 0$ and $\xi_\lambda - \xi_0 = 0$ on $\partial B^n$  using \eqref{eq:small-identity-alpha} and  \eqref{eq:omega-delta} yields that:
\begin{align}
    \|d(\xi_\lambda - \xi_0)\|_{W^{1,(p,1)}} \lesssim\| d(\xi_\lambda - \xi_0)\|_{W^{1,(p,1)}} \lesssim \alpha + \delta =: \beta. \label{eq:betalefttoshow}
\end{align}

Note, that \eqref{eq:small-identity-alpha} immediately implies
\begin{equation}
    \begin{aligned}
        \|d(P_\lambda - P_0)\|_{W^{1,(p,1)}} \le C \|d(Q_\lambda - \mathrm{Id}_m)\|_{W^{1,(p,1)}} < \alpha. \label{eq:alpha-left-to-show}
    \end{aligned}
\end{equation}

Choosing $\varepsilon, \alpha, \beta > 0$ sufficiently small such that $C \varepsilon + \alpha + \beta \le \delta$, the triangle inequality yields
\begin{equation}
    \begin{aligned}
        \|dP_\lambda\|_{L^{(n,q)}} + \|d\xi_\lambda\|_{L^{(n,q)}} 
        &\le \|dP_0\|_{L^{(n,q)}} + \|d\xi_0\|_{L^{(n,q)}} \\
        &\hspace{1cm}+ \|d(P_\lambda - P_0)\|_{L^{(n,q)}} + \|d(\xi_\lambda - \xi_0)\|_{L^{(n,q)}} \\
        &\le C\varepsilon + \alpha + \beta \le \delta.
    \end{aligned}
\end{equation}
Therefore, $(P_\lambda, \xi_\lambda)$ satisfies the assumption of Lemma \ref{lm:lemma-to-conclude}, which implies the a-priori bounds \eqref{eq:apriori-bound-better-lemma1} and \eqref{eq:apriori-bound-better-lemma2} for $(P_\lambda, \xi_\lambda)$. 

This proves $\Omega_0 + \omega \in \mathcal{U}_{\varepsilon,C}$, and hence that $\mathcal{U}_{\varepsilon,C}$ is open in $\mathcal{V}_\varepsilon$.

\textbf{Conclusion.}
Step  1 -- Step 4 implies $\mathcal{U}_{\varepsilon,C} = \mathcal{V}_\varepsilon$ and this finishes the proof of Lemma \ref{lm:better-regularity-space}.

\end{proof}

It remains to prove the following auxiliary Lemma.

\begin{lemma} \label{lm:lemma-to-conclude}
There exist $C(n) > 0$, $\varepsilon>0$ and $\delta > 0$ such that for every $P \in W^{2,(p,1)}(B^n, \mathrm{SO}(m))$ and $\xi \in W^{2,(p,1)}(B^n, \mathfrak{so}(m)\otimes \wedge^{n-2})$ satisfying \eqref{eq:gauge-system-better} for some $\Omega \in W^{1,(p,1)}(B^n, \mathfrak{so}(m)\otimes \wedge^1)$ with $\Vert \Omega \Vert_{\lorentzl}<\varepsilon$, if
\[
\|dP\|_{L^{(n,q)}} + \|d\xi\|_{L^{(n,q)}} \le \delta,
\]
then \eqref{eq:apriori-bound-better-lemma1} and \eqref{eq:apriori-bound-better-lemma2} are satisfied.
\end{lemma}
\begin{proof}

We decompose $\xi = v + w$, where $v, w \in W^{2,(p,1)}(B^n, \wedge^{n-2})$ satisfy:
\begin{equation}\label{eq:u_system}
\begin{cases}
    \Delta v = \ast(dP^{-1} \wedge dP) & \text{in } B^n, \\
    v = 0 & \text{on } \partial B^n,
\end{cases}
\qquad
\begin{cases}
    \Delta w = \ast d(P^{-1}\Omega P) & \text{in } B^n, \\
    w = 0 & \text{on } \partial B^n.
\end{cases}
\end{equation}

\vspace*{0.5cm}

\textbf{Estimates containing $v$:}

\textit{$L^{(n,q)}$ Bound for $dv$:}
In order to exploit the div-curl structure, we use that $dv\in W^{1,(p,1)}\hookrightarrow L^{(n,1)} \hookrightarrow L^{(n,q)}$ (for $1\leq q\leq 2$) and hence
\begin{equation}
    \begin{aligned}
       \Vert dv \Vert_{L^{(n,q)}(B^n)} &\lesssim     \Vert dv \Vert_{L^{(n,1)}(B^n)} \lesssim \Vert \Delta v \Vert_{L^{(\frac{n}{2},1)}(B^n)} \lesssim \Vert dP^{-1} \Vert_{L^{(n,q)}} \Vert dP \Vert_{L^{(n,q')}} \\
       & \overset{L^{(n,q)} \subset L^{(n,q')}}{\leq} C \delta \Vert dP \Vert_{L^{(n,q)}(B^n)}.
    \end{aligned} \label{eq:lorentz-1-for-v1}
\end{equation}

\textit{$W^{1,(p,1)}$ Bound for $dv$:} 
Using Calderón-Zygmund estimates together with the Lorentz-Hölder inequality and the Sobolev embedding $W^{1,(p,1)}(B^n) \hookrightarrow L^{(\frac{np}{n-p},1)}(B^n) \hookrightarrow L^{(\frac{np}{n-p},q')}(B^n)$, we get
\begin{equation}
\begin{aligned}
\Vert dv \Vert_{W^{1,(p,1)}}
&\lesssim \Vert \Delta v\Vert_{L^{(p,1)}} \lesssim \Vert dP \Vert_{L^{(n,q)}} \Vert dP \Vert_{L^{(\frac{np}{n-p},q')}} \lesssim \delta \Vert dP \Vert_{W^{1,(p,1)}}.
\end{aligned}  \label{eq:lorentz-2-for-v1}
\end{equation}

\vspace*{0.5cm}

\textbf{Estimates containing $w$:}

\textit{$L^{(n,q)}$ Bound for $dw$:} Again by standard elliptic estimates  and $|P| \le 1$ we have
\begin{equation}
    \|d w\|_{L^{(n,q)}} \le C \|P^{-1}\Omega P\|_{L^{(n,q)}} \le C \|\Omega\|_{L^{(n,q)}}. \label{eq:lorentz-1-for-w1}
\end{equation}

\textit{$W^{1,(p,1)}$ Bound for $dw$:} Applying the product rule yields
\begin{equation}
    d(P^{-1}\Omega P) = dP^{-1}\Omega P + P^{-1}(d\Omega)P + P^{-1}\Omega dP.
\end{equation}
Using $d\Omega \in L^{(p,1)}$ together with $W^{1,(p,1)} \hookrightarrow L^{(\frac{np}{n-p},q')}$:
\begin{equation} \begin{aligned}
    \|d(P^{-1}\Omega P)\|_{L^{(p,1)}} &\le C \Big( \|d\Omega\|_{L^{(p,1)}} + \|dP\|_{L^{(n,q)}} \|\Omega\|_{L^{(\frac{np}{n-p},q')}} \Big)  \\
    &\le C \|\Omega\|_{W^{1,(p,1)}} \Big( 1 +   \|dP\|_{L^{(n,q)}} \Big).
\end{aligned} \end{equation} 
Again, by standard elliptic estimates we obtain:
\begin{equation}
    \|dw\|_{W^{1,(p,1)}} \le C \|\Delta w\|_{L^{(p,1)}} \leq C \|\Omega\|_{W^{1,(p,1)}} \Big( 1 +   \|P\|_{W^{1,(p,1)}} \Big).\label{eq:lorentz-2-for-w1}
\end{equation}

\vspace*{0.5em}
\textbf{Conclusion for $\xi$:}

Note that by \eqref{eq:gauge-system-better}, we have the expression 
\begin{align}
    \ast d\xi =dP^\top P+ P^\top \Omega P \qquad \Leftrightarrow \qquad dP^\top= \ast d\xi P^\top - P^\top \Omega . \label{eq:gauge-formula-rewritten}
\end{align}
This leads to 
\begin{equation}
\begin{aligned}
 \Vert dP \Vert_{\lorentzl} &\lesssim \Vert d\xi \Vert_{\lorentzl} + \Vert \Omega \Vert_{\lorentzl}.  \\
    \Vert dP \Vert_{W^{1,(p,1)}} &\lesssim  \Vert d\xi \Vert_{W^{1,(p,1)}} +  \Vert \Omega \Vert_{W^{1,(p,1)}}. \label{eq:estimate-for-p-by-defintion11}
 \end{aligned}
\end{equation}
So once we established the bounds \eqref{eq:apriori-bound-better-lemma1} and \eqref{eq:apriori-bound-better-lemma2} for $d\xi$, we can conclude the proof by using \eqref{eq:estimate-for-xi}. 

\vspace{0.5cm}
By combining \eqref{eq:lorentz-1-for-v1} and \eqref{eq:lorentz-1-for-w1}, we have
\begin{equation}
    \begin{aligned}
        \Vert d\xi \Vert_{L^{(n,q)}} &\leq   \Vert dv \Vert_{L^{(n,q)}} +   \Vert dw \Vert_{L^{(n,q)}} \lesssim  \delta \Vert dP \Vert_{\lorentzl} + \Vert \Omega \Vert_{\lorentzl} \\
       & \overset{\eqref{eq:estimate-for-p-by-defintion11}}{\leq} C_1 \delta \Vert d\xi \Vert_{\lorentzl} + C_2(1+\delta) \Vert \Omega \Vert_{\lorentzl}.
    \end{aligned}
\end{equation}
Choosing $\delta$ sufficiently small such that $C_1\delta <\frac{1}{2}$,  absorbing the term containing $d\xi$ to the LHS, we get \eqref{eq:apriori-bound-better-lemma2}.

In order to show \eqref{eq:apriori-bound-better-lemma1}, combining \eqref{eq:lorentz-2-for-v1} and \eqref{eq:lorentz-2-for-w1} yields 
\begin{equation}
    \begin{aligned}
        \Vert d\xi \Vert_{W^{1,(p,1)}} &\leq   \Vert dv \Vert_{W^{1,(p,1)}} +   \Vert dw \Vert_{W^{1,(p,1)}} \lesssim \delta \Vert dP \Vert_{W^{1,(p,1)}} + \Vert \Omega \Vert_{W^{1,(p,1)}} \\
        & \overset{\eqref{eq:estimate-for-p-by-defintion11}}{\leq} C_3 \delta \Vert d\xi\Vert_{W^{1,(p,1)}} + C_4\Vert \Omega \Vert_{W^{1,(p,1)}}. \label{eq:estimate-for-xi}
    \end{aligned}
\end{equation}
Again, choose  $\delta $ small enough, such that also $C_3 \delta <\frac{1}{2}$ holds, we can conclude \eqref{eq:apriori-bound-better-lemma1}.

\end{proof}

\begin{lemma} \label{lm:lemma-of-harmonic}
Let  $1 \le p < +\infty$, and let $B^n \subset \R^n$.
Then, for any function $f \in L^p(B^n)$ which is a solution of
\[
-\Delta f = 0 \quad \text{on } B^n,
\]
the function
\[
r \mapsto \frac{1}{r^n} \int\limits_{B_r} |f|^p \, dx^n,
\]
is increasing.
\end{lemma}

\begin{proof}
We follow the proof in  \cite[Lemma 3.3.12 ]{movingframes}.

Setting $v = |f|^p$, we first compute the Laplacian of $v$. Since $f$ is harmonic ($\Delta f = 0$), a direct calculation gives
\begin{align}
    \operatorname{div}(\nabla v) = \operatorname{div}\big(p |f|^{p-2} f \nabla f\big) = \underbrace{p(p-1)|f|^{p-2}}_{\geq 0} \cdot \underbrace{\vert \nabla f \vert^2}_{\geq 0} + p |f|^{p-2} f \underbrace{\Delta f}_{=0} \geq 0.
\end{align}
Therefore, $\Delta v \geq 0$, so $v = |f|^p$ is subharmonic. 

Integrating this inequality over the ball $B_r \subset \mathbb{R}^n$ and applying Stokes' theorem, we obtain
\begin{align}
    0 \leq \int\limits_{B_r} \Delta v \, dx = \int\limits_{\partial B_r} \frac{\partial v}{\partial r}(z) \; d\mu(z) = r^{n-1} \frac{d}{dr} \int\limits_{\partial B^n} v(r\theta) \; d\mu(\theta), \label{eq: non-decreasing}
\end{align}
where we parametrized the boundary by $z = r\theta$ for $\theta \in \partial B^n$ and used the fact that the surface measure scales as $d\mu(z) = r^{n-1} d\mu(\theta)$. 

Equation \eqref{eq: non-decreasing} implies that 
\begin{align}
    r \mapsto \int\limits_{\partial B^n} |f(r\theta)|^p \; d\mu(\theta) \label{eq: gain-1}
\end{align} 
is non-decreasing. It follows that
\begin{align}
\int\limits_{B_r} |f(z)|^p\, dz
&= \int\limits_0^r \rho^{n-1} \, d\rho
  \int\limits_{\partial B^n} |f(\rho\theta)|^p\, d\mu(\theta) 
\overset{\eqref{eq: gain-1}}{\le}
\int\limits_0^r \rho^{n-1} \, d\rho
  \int\limits_{\partial B^n} |f(r\theta)|^p\, d\mu(\theta) \\
&= \frac{r^n}{n} 
  \int\limits_{\partial B^n} |f(r\theta)|^p\, d\mu(\theta) 
= \frac{r^n}{n} \frac{1}{r^{n-1}} \int\limits_{\partial B_r} |f(x)|^p \, d\mu(x) \\
&= \frac{r}{n} \frac{d}{dr}
  \int\limits_{B_r} |f(x)|^p\, dx. \label{eq: ode-part1}
\end{align}
Defining $\displaystyle F(r):= \int\limits_{B_r} |f(x)|^p \; dx$, equation \eqref{eq: ode-part1} can be rewritten as
\begin{align}
    F(r) \leq \frac{r}{n}F'(r) \quad \Leftrightarrow \quad 0 \leq \frac{F'(r)}{r^n}- \frac{nF(r)}{r^{n+1}}= \frac{d}{dr} \Big( \frac{F(r)}{r^n} \Big).
\end{align}
Integrating this differential inequality, we conclude that
\[
r \mapsto \frac{1}{r^n} \int\limits_{B_r} |f(x)|^p\, dx
\]
is non-decreasing.
\end{proof}

\section*{Acknowledgements}

The author would like to thank Francesca Da Lio and Tristan Rivière for their enriching discussions, ideas and feedback on previous versions of the paper and continued support. 
\vspace{0.2cm}

The author would like to thank Dorian Martino for his advice and discussion on the AI generated Section \ref{sec:counterexample}.
\vspace{0.5cm}

\footnotesize
\begin{description}
    \item[On the Use of AI] All Sections, except Section \ref{sec:counterexample} do not contain any proofs, strategies or ideas generated by AI. Upon completion of the paper, the author used ChatGPT-5.6 Sol Pro to show the optimality in Section \ref{sec:counterexample}. This Section was only verified by the author. The author used Gemini to find literature and typos.
\end{description}

\normalsize